\documentclass[oneside, a4paper, 10pt]{article}
\usepackage[colorlinks = true,
            linkcolor = black,
            urlcolor  = blue,
            citecolor = black]{hyperref}
\usepackage{amsmath, amsthm, amssymb, amsfonts, amscd}
\usepackage[left=3cm,right=3cm,top=3cm,bottom=3cm,a4paper]{geometry}
\usepackage{mathtools}
\usepackage{thmtools}
\usepackage{graphicx}
\usepackage{cleveref}
\usepackage{mathrsfs}
\usepackage{titling}
\usepackage{url}
\usepackage[nosort, noadjust]{cite}
\usepackage{enumitem}
\usepackage{tikz-cd}
\usepackage{multirow}
\usepackage{fancyhdr}
\usepackage{sectsty}
\usepackage{longtable}
\usepackage{authblk}
\usepackage[title]{appendix}
\DeclareGraphicsExtensions{.pdf,.png,.jpg}

\newcommand{\im}{{\operatorname{im}}}
\newcommand{\cok}{{\operatorname{cok}}}

\newcommand{\Gal}{\operatorname{Gal}}
\newcommand{\Hom}{\operatorname{Hom}}

\newcommand{\Aut}{\operatorname{Aut}}
\newcommand{\GL}{\operatorname{GL}}

\newcommand{\rank}{\operatorname{rank}}

\newcommand{\diag}{\operatorname{diag}}
\newcommand{\Sur}{\operatorname{Sur}}

\newcommand{\M}{\operatorname{M}}

\newcommand{\Z}{\mathbb{Z}}
\newcommand{\Q}{\mathbb{Q}}
\newcommand{\R}{\mathbb{R}}

\newcommand{\F}{\mathbb{F}}

\newcommand{\PP}{\mathbb{P}}
\newcommand{\EE}{\mathbb{E}}
\newcommand{\OO}{\mathcal{O}}
\newcommand{\CC}{\mathcal{C}}
\newcommand{\DD}{\mathcal{D}}
\newcommand{\GG}{\mathcal{G}}
\newcommand{\LL}{\mathcal{L}}
\newcommand{\calA}{\mathcal{A}}

\newcommand{\bfe}{\mathbf{e}}
\newcommand{\bfk}{\mathbf{k}}

\newcommand{\mfm}{\mathfrak{M}}
\newcommand{\WG}{\widetilde{G}}

\theoremstyle{definition}
\newtheorem{theorem}{Theorem}[section]
\newtheorem{proposition}[theorem]{Proposition}

\newtheorem{lemma}[theorem]{Lemma}

\newtheorem*{conjecture*}{Conjecture}
\newtheorem{remark}[theorem]{Remark}
\newtheorem{definition}[theorem]{Definition}

\makeatletter 
\newcommand\semilarge{\@setfontsize\semilarge{11}{13.2}}
\makeatother

\title{\semilarge{\textbf{MIXED MOMENTS AND THE JOINT DISTRIBUTION OF RANDOM GROUPS}}}
\author{\normalsize{JUNGIN LEE} }
\date{}

\newcommand\shorttitle{JOINT DISTRIBUTION OF RANDOM GROUPS}
\newcommand\authors{JUNGIN LEE}

\sectionfont{\centering \large}
\subsectionfont{\normalsize}
\makeatletter
\renewcommand{\@seccntformat}[1]{\csname the#1\endcsname.\quad}
\makeatother

\renewenvironment{abstract}
 {\quotation\small\noindent\rule{\linewidth}{.5pt}\par\smallskip
  {\centering\bfseries\abstractname\par}\medskip}
 {\par\noindent\rule{\linewidth}{.5pt}\endquotation}
 
\AtEndDocument{%
  \bigskip
  \footnotesize

  \textsc{Jungin Lee, Department of Mathematics, Ajou University, Suwon 16499, Republic of Korea}\par\nopagebreak
  \textit{E-mail address:} \texttt{jileemath@ajou.ac.kr} }
\begin{document}
\maketitle
\vspace{-18mm}

\begin{abstract}
We study the joint distribution of random abelian and non-abelian groups. In the abelian case, we prove several universality results for the joint distribution of the multiple cokernels for random $p$-adic matrices. In the non-abelian case, we compute the joint distribution of random groups given by the quotients of the free profinite group by random relations. In both cases, we generalize the known results on the distribution of the cokernels of random $p$-adic matrices and random groups. Our proofs are based on the observation that mixed moments determine the joint distribution of random groups, which extends the works of Wood for abelian groups and Sawin for non-abelian groups.
\end{abstract}

\section{Introduction} \label{Sec1}

The moment problem is to determine whether a probability distribution is uniquely determined by its moments. 
A version of the classical moment problem concerns the distribution of a random variable $X \in \R$ and its moments given by the expected values $m_k = \EE(X^k)$ for $k \in \Z_{\geq 0}$. 
It is well-known that if the moments of $X$ are finite and satisfy the Carleman's condition
$$
\sum_{k=1}^{\infty} m_{2k}^{-\frac{1}{2k}} = \infty,
$$
then $X$ is uniquely determined by its moments. 

This can be extended to the joint distribution of (not necessarily independent) random variables $X_1, \hdots, X_r \in \R$. Their mixed moments are defined to be the expected values $\EE(X_1^{k_1} \cdots X_r^{k_r})$ ($k_1, \hdots, k_r \in \Z_{\geq 0}$). There is also a natural generalization of Carleman's condition to the mixed moments which is a sufficient condition for the uniqueness of the joint distribution \cite[Theorem 2.3]{deJ03}.

There have been a lot of work to determine the distribution of random groups from their moments. An important example is the distribution of the cokernels of random matrices over the $p$-adic integers $\Z_p$. Let $\M_{m \times n}(R)$ be the set of $m \times n$ matrices over the ring $R$, $\M_n(R) = \M_{n \times n}(R)$ and $\GG_p$ be the set of isomorphism classes of finite abelian $p$-groups. Friedman and Washington \cite{FW89} proved that if $A_n \in \M_n(\Z_p)$ is a uniform random matrix for each $n$ and $H \in \GG_p$, then
\begin{equation} \label{eq1a}
\lim_{n \rightarrow \infty} \PP(\cok(A_n) \cong H) = \frac{\prod_{k=1}^{\infty}(1-p^{-k})}{\left| \Aut(H) \right|}.  
\end{equation}
Wood \cite{Woo19}, and Nguyen and Wood \cite{NW22a} proved that the same result holds for larger classes of random matrices $A_n$ whose entries are independent and not too concentrated. 
In order to state their theorem, we take the following definition from \cite{Woo19}. For $\varepsilon > 0$, a random variable $x \in \Z_p$ is $\varepsilon$-\textit{balanced} if $\PP(x \equiv r \,\, (\text{mod } p)) \leq 1 - \varepsilon$ for every $r \in \Z / p \Z$. A random matrix $A \in \M_{m \times n}(\Z_p)$ is $\varepsilon$-\textit{balanced} if its entries are independent and $\varepsilon$-balanced. 

\begin{theorem} \label{thm1a}
(\cite[Theorem 4.1]{NW22a}) Let $u \geq 0$ be an integer, $H \in \GG_p$ and $(\alpha_n)_{n \geq 1}$ be a sequence of real numbers such that for every $\Delta > 0$, we have $\alpha_n \geq \frac{\Delta \log n}{n}$ for all sufficiently large $n$. Let $A_n \in \M_{n \times (n+u)}(\Z_p)$ be an $\alpha_n$-balanced random matrix for each $n$.  Then we have
\begin{equation} \label{eq1b}
\lim_{n \rightarrow \infty} \PP(\cok(A_n) \cong H) = \frac{\prod_{k=1}^{\infty} (1-p^{-k-u})}{\left| H \right|^{u} \left| \Aut(H) \right| }.
\end{equation}
\end{theorem}

There are analogues of Theorem \ref{thm1a} for various types of random $p$-adic matrices. For every $\varepsilon > 0$, the universality holds for the cokernels of $\varepsilon$-balanced random symmetric matrices over $\Z_p$ \cite[Theorem 1.3]{Woo17}, the cokernels of $\varepsilon$-balanced random alternating matrices over $\Z_p$ \cite[Theorem 1.13]{NW22b} and the cokernels of $\varepsilon$-balanced random Hermitian matrices over the ring of integers $\OO$ of a quadratic extension of $\Q_p$ \cite[Theorem 1.7]{Lee23b}. 
We also note that the lower bound $\alpha_n \geq \frac{\Delta \log n}{n}$ given in Theorem \ref{thm1a} is optimal up to constants \cite[p. 20]{Woo22}. Indeed, if each entry of $A_n$ is chosen to be $0$ with a probability $1 - \frac{\log n}{n}$, then the probability $p_n$ that $A_n$ has a row of all zeroes satisfies
$$
\lim_{n \rightarrow \infty} p_n
= 1 - \lim_{n \rightarrow \infty} \left ( 1 - \left ( 1 - \frac{\log n}{n} \right )^{n+u} \right )^n
= 1 - e^{-1}.
$$
This implies that $\liminf_{n \rightarrow \infty} \PP(\cok(A_n) \text{ is infinite}) \geq 1 - e^{-1} > 0$.

It is very difficult to obtain the distribution of the cokernel of $A_n$ by a direct computation, unless the entries of $A_n$ are equidistributed with respect to Haar measure. 
This shows the necessity of the use of the moments of random groups in the proof of Theorem \ref{thm1a}. 
For a given finite group $H$, the $H$-\textit{moment} of a random finite group $X$ is defined by the expected value $\EE(\# \Sur(X, H))$ of the number of surjective homomorphisms from $X$ to $H$. The proof of Theorem \ref{thm1a} can be divided into two steps: first we verify that the $n \rightarrow \infty$ limit of the moment $\EE(\# \Sur(\cok(A_n), H))$ is $\left| H \right|^{-u}$ for every $H \in \GG_p$ \cite[Theorem 4.12]{NW22a}, then we prove that these moments determine the unique distribution \cite[Theorem 3.1]{Woo19}.

Now we introduce a different kind of generalization of the equation (\ref{eq1a}). Friedman and Washington \cite{FW89} proved that if $A_n$ is a uniform random matrix in $\GL_n(\Z_p)$ for each $n$ and $H \in \GG_p$, then
\begin{equation} \label{eq1c}
\lim_{n \rightarrow \infty} \PP (\cok(A_n - I_n) \cong H) 
= \frac{\prod_{k=1}^{\infty}(1-p^{-k})}{\left| \Aut(H) \right|}.    
\end{equation}
($I_n$ denotes the $n \times n$ identity matrix.) They proved it as a heuristic evidence of the function field analogue of the Cohen-Lenstra heuristics. 
Motivated by this result, Cheong and Huang \cite{CH21} predicted the limiting joint distribution of the cokernels $\cok(P_j(A_n))$ ($1 \leq j \leq r$) where $P_1, \hdots, P_r \in \Z_p[t]$ are monic polynomials with some assumptions and $A_n \in \M_n(\Z_p)$ is a uniform random matrix for each $n$. 
Based on elementary probabilistic methods, the author \cite{Lee23a} proved the conjecture of Cheong and Huang. 

\begin{theorem} \label{thm1b}
Let $A_n \in \M_n(\Z_p)$ be a uniform random matrix for each $n$.
\begin{enumerate}
\item (\cite[Theorem 2.1]{Lee23a}) Let $P_1, \hdots, P_r \in \Z_p[t]$ be monic polynomials whose mod $p$ reductions in $\F_p[t]$ are distinct and irreducible. Also let $H_j$ be a finite module over $R_j := \Z_p[t]/(P_j (t))$ for each $1 \leq j \leq r$. Then we have
\begin{equation*}
\lim_{n \rightarrow \infty} \PP \begin{pmatrix}
\cok(P_j(A_n)) \cong H_j \\
\text{for } 1 \leq j \leq r
\end{pmatrix} = \prod_{j=1}^{r} \frac{\prod_{k=1}^{\infty} (1-p^{-k \deg(P_j)})}{\left| \Aut_{R_j}(H_j) \right|}.
\end{equation*}
    
\item (\cite[Theorem 4.1]{Lee23a}) Let $H_1, H_2 \in \GG_p$ and $(B_n)_{n \geq 1}$ be a sequence of matrices such that $B_n \in \M_n(\Z_p)$ and $\lim_{n \rightarrow \infty} (r_p(B_n) - \log_p n) = \infty$, where $r_p(B_n)$ is the rank of the matrix $\overline{B_n} \in \M_n(\F_p)$ defined by the reduction modulo $p$ of $B_n$. Then we have
\begin{equation*} 
\lim_{n \rightarrow \infty} \PP \begin{pmatrix}
\cok(A_n) \cong H_1 \text{ and} \\
\cok(A_n+B_n) \cong H_2
\end{pmatrix} = \prod_{i=1}^{2} \frac{\prod_{k=1}^{\infty} (1-p^{-k})}{\left| \Aut(H_i) \right| }.
\end{equation*}
\end{enumerate}
\end{theorem}

There are some recent works related to the above result. Cheong and Kaplan \cite[Theorem 1.1]{CK22} independently proved Theorem \ref{thm1b}(1) under the assumption that $\deg(P_j) \leq 2$ for each $j$ by a different method. Cheong, Liang and Strand \cite[Theorem 1.1]{CLS23} computed the probability $\PP (\cok(P(A_n)) \cong H)$ for a monic polynomial $P \in \Z_p[t]$, a $\Z_p[t]/(P(t))$-module $H$ and a fixed $n$. Van Peski \cite[Theorem 1.4]{VP23} computed the joint distribution of
$$
\cok(A_1), \, \cok(A_2A_1), \hdots, \, \cok(A_r \cdots A_1)
$$
for a fixed $n \geq 1$ and uniform random matrices $A_1, \hdots, A_r \in \M_n(\Z_p)$ using explicit formulas for certain skew Hall-Littlewood polynomials. 

It is natural to ask whether we can compute the joint distribution of the multiple cokernels for a sequence of random matrices $(A_n)_{n \geq 1}$ given as in Theorem \ref{thm1a}. 
In order to compute the joint distribution for non-uniform random matrices, we introduce the mixed moments of random groups.
Let $X_1, \hdots, X_r$ be (not necessarily independent) random finite groups. The \textit{mixed moments} of $X_1, \hdots, X_r$ are defined to be the expected values
$$
\EE \left ( \prod_{k=1}^{r} \# \Sur(X_k, G_k) \right ) 
$$
for finite groups $G_1, \hdots, G_r$. 
The following theorem can be easily deduced from Theorem \ref{thm21c} as in \cite[Corollary 9.2]{Woo17}, by taking $a = \prod_{p \in P} p^{e_p+1}$ when $(\prod_{p \in P} p^{e_p}) G_k = 0$ for all $1 \le k \le r$. The proof of Theorem \ref{thm21c} is given in Section \ref{Sub21}. By taking $P$ to be the set $\left\{ p \right\}$, we conclude that mixed moments determine the joint distribution of $X_1, \hdots, X_r \in \GG_p$ if they are not too large. 

\begin{theorem} \label{thm1c}
Let $P$ be a finite set of primes, $\calA$ be the set of finite abelian groups whose order is a product of powers of primes in $P$ and $Y = (Y_1, \hdots, Y_r)$, $X_n = (X_{n, 1}, \hdots, X_{n, r})$ ($n \geq 1$) be $r$-tuples of random groups in $\calA$. Suppose that for every $G_1, \hdots, G_r \in \calA$, we have
\begin{equation*}
\lim_{n \rightarrow \infty} \EE(\prod_{k=1}^{r} \# \Sur(X_{n, k}, G_k))
= \EE(\prod_{k=1}^{r} \# \Sur(Y_k, G_k))
= O \left ( \prod_{k=1}^{r} m(G_k) \right ).
\end{equation*}
(The number $m(G_k)$ is defined just before Theorem \ref{thm21c}.) Then for every $H_1, \hdots, H_r \in \calA$, 
$$
\lim_{n \rightarrow \infty} \PP (X_{n, k} \cong H_k \text{ for } 1 \leq k \leq r)
= \PP (Y_k \cong H_k \text{ for } 1 \leq k \leq r).
$$
\end{theorem}

\begin{remark} \label{rmk1new1}
While this paper was close to completion, we became aware of a recent preprint by Nguyen and Van Peski \cite{NVP23} which also introduced the mixed moments of abelian groups to compute the joint distribution of the cokernels of random matrix products. They proved universality for the limiting joint distribution of $\cok(A_1), \cok(A_2A_1), \hdots, \cok(A_r \cdots A_1)$ where $A_1, \hdots, A_r \in \M_n(\Z_p)$ are $\varepsilon$-balanced random matrices for $\varepsilon > 0$ \cite[Theorem 1.1]{NVP23}, using a result similar to the above theorem \cite[Theorem 9.1]{NVP23}. 
We also note that the above result can be deduced from the proof of \cite[Theorem 6.11]{WW21}, as mentioned in \cite[Section 1.3]{SW22b}.
\end{remark}

In Section \ref{Sec3}, we provide three universality results for the joint distribution of the cokernels of random $p$-adic matrices using Theorem \ref{thm1c}. First we provide a combination of Theorem \ref{thm1a} and the $\deg (P_j) = 1$ case of Theorem \ref{thm1b}(1).

\begin{theorem} \label{thm1d}
(Theorem \ref{thm31c}) Let $t_1, \hdots, t_r$ be integers such that $p \nmid t_j - t_{j'}$ for each $j \neq j'$ and $A_n \in \M_n(\Z_p)$ be as in Theorem \ref{thm1a}. Then we have
\begin{equation*} 
\lim_{n \rightarrow \infty} \PP \begin{pmatrix}
\cok(A_n+t_jI_n) \cong H_j \\
\text{for } 1 \leq j \leq r
\end{pmatrix} = \prod_{j=1}^{r} \frac{\prod_{k=1}^{\infty} (1 - p^{-k})}{\left| \Aut(H_j) \right|}
\end{equation*}
for every $H_1, \hdots, H_r \in \GG_p$.
\end{theorem}

The next theorem is a generalization of Theorem \ref{thm1b}(2). Note that we have removed the term $\log_p n$ which appears in Theorem \ref{thm1b}(2).

\begin{theorem} \label{thm1e}
(Theorem \ref{thm31e}) Let $u \geq 0$ be an integer, $A_n \in \M_{n \times (n+u)}(\Z_p)$ be as in Theorem \ref{thm1a} and $(B_n)_{n \geq 1}$ be a sequence of matrices such that $B_n \in \M_{n \times (n+u)}(\Z_p)$ and $\lim_{n \rightarrow \infty} r_p(B_n) = \infty$ where $r_p(B_n)$ is the rank of $\overline{B_n} \in \M_n(\F_p)$. Then we have
\begin{equation*} 
\lim_{n \rightarrow \infty} \PP \begin{pmatrix}
\cok(A_n) \cong H_1 \text{ and} \\
\cok(A_n+B_n) \cong H_2
\end{pmatrix} = \prod_{i=1}^{2} \frac{\prod_{k=1}^{\infty} (1-p^{-k-u})}{\left| H_i \right|^{u} \left| \Aut(H_i) \right| }
\end{equation*}
for every $H_1, H_2 \in \GG_p$.
\end{theorem}

\begin{remark} \label{rmk1f}
If two events $\cok(A_n)=0$ and $\cok(A_n+B_n)=0$ are asymptotically independent, then we have $\lim_{n \rightarrow \infty} r_p(B_n) = \infty$ by \cite[Proposition 4.5]{Lee23a}. This shows that Theorem \ref{thm1e} is best possible in the sense that we cannot weaken the condition $\lim_{n \rightarrow \infty} r_p(B_n) = \infty$.
\end{remark}

Finally, we consider the joint distribution of $\cok(A)$ and $\cok(A+pI_n)$. Unlike the first two applications, the cokernels $\cok(A)$ and $\cok(A+pI_n)$ have the same $p$-rank so they cannot be asymptotically independent. 
For a finite abelian $p$-group $G$, denote its $p$-rank by $r_p(G) := \rank_{\F_p}(G / pG)$. Define $c_r(p) := \prod_{k=1}^{r} (1-p^{-k})$ and $c_{\infty}(p) := \prod_{k=1}^{\infty} (1-p^{-k})$. 

\begin{theorem} \label{thm1g}
(Theorem \ref{thm32f}) Let $A_n \in \M_n(\Z_p)$ be as in Theorem \ref{thm1a}. Then we have
\begin{equation*} 
\lim_{n \rightarrow \infty} \PP \begin{pmatrix}
\cok(A_n) \cong H_1 \text{ and} \\
\cok(A_n+pI_n) \cong H_2
\end{pmatrix} = \left\{\begin{matrix}
0 & (r_p(H_1) \neq r_p(H_2)) \\
\frac{p^{r^2} c_{\infty}(p)c_{r}(p)^2}{\left| \Aut(H_1) \right| \left| \Aut(H_2) \right|} & (r_p(H_1) = r_p(H_2) = r)
\end{matrix}\right.
\end{equation*}
for every $H_1, H_2 \in \GG_p$.
\end{theorem}

Our proof of Theorem \ref{thm1g} consists of two parts. First, we prove the theorem under the assumption that $A_n$ are uniform random matrices based on the techniques used in \cite{Lee23a}. 
Second, we extend this to the general case by computing the mixed moments of $(\cok(A_n), \cok(A_n + pI_n))$ and applying Theorem \ref{thm1c}. 

In the rest of this section, we consider the distribution of random non-abelian groups. First we introduce the work of Liu and Wood \cite{LW20} on a random group given by the quotient of the free profinite group by random relations. 
For a finite set $\CC$ of finite groups, let $\overline{\CC}$ be the smallest set of finite groups containing $\CC$ which is closed under taking quotients, subgroups and finite direct products. 
We say a profinite group $G$ is \textit{level}-$\CC$ if $G \in \overline{\CC}$. 
Let $\GG$ be the set of isomorphism classes of profinite groups $G$ such that $G^{\CC}$ is finite for every finite set $\CC$ of finite groups, where $G^{\CC}$ is the inverse limit of level-$\CC$ quotients of $G$. For a profinite group $G \in \GG$, we have $G \in \overline{\CC}$ if and only if $G^{\CC} = G$ (\cite[p. 127]{LW20}).
We endow $\GG$ with a topology whose basic open sets are of the form $U_{\CC, H} := \left\{ G \in \GG \mid G^{\CC} \cong H \right\}$, where $\CC$ is a finite set of finite groups and $H$ is a finite group. In this paper, every measure on $\GG$ is assumed to be a Borel measure. For $g_1, g_2, \hdots, g_r \in F_n$, $\left< g_1, g_2, \hdots, g_r \right>$ denotes the closed normal subgroup of $F_n$ generated by the elements $g_1, g_2, \hdots, g_r$. 

\begin{theorem} \label{thm1h}
(\cite[Theorem 1.1]{LW20}) Let $u \geq 0$ be an integer, $F_n$ be the free profinite group on $n$ generators and $r_1, \hdots, r_{n+u}$ be independent Haar random elements of $F_n$. 
Then there is a probability measure $\mu_u$ on $\GG$ such that the distributions of $F_n/\left< r_1, \hdots, r_{n+u} \right>$ weakly converge in distribution to $\mu_u$ as $n \rightarrow \infty$.
\end{theorem}

We refer \cite[Equation (3.2)]{LW20} for an explicit formula for $\mu_u$ on each basic open set $U_{\CC, H}$.

Random non-abelian groups can be applied to the study of the distribution of the Galois group of maximal unramified extension of global fields. Let $\Gamma$ be a finite group, $\F_q$ be a finite field of order $q$ such that $(q, \left| \Gamma \right|) = 1$ and $K/\F_q[t]$ be a $\Gamma$-extension which splits completely at $\infty$. Define $K^{\#}$ to be the maximal unramified extension of $K$ such that every finite subextension $L$ of $K^{\#}/K$ satisfies the condition $([L:K], q(q-1)\left| \Gamma \right|) = 1$ and
$K^{\#}/K$ splits completely at places over $\infty$. Then the Galois group $\Gal(K^{\#}/K)$ has a $\Gamma$-group structure by conjugation. (A $\Gamma$-\textit{group} is a profinite group with a continuous action of $\Gamma$.) 

Let $\CC$ be a finite set of finite $\Gamma$-groups such that $(\left| \CC \right|, \left| \Gamma \right|) = 1$. Here $\left| \CC \right|$ denotes the least common multiple of the orders of elements of $\CC$. 
Liu, Wood and Zureick-Brown \cite[Theorem 1.4]{LWZ19} proved that the limit of the $H$-moment of the random Galois group $\Gal(K^{\#}/K)$ is given by $[H: H^{\Gamma}]^{-1}$. 
Sawin proved that these moments determine the distribution of a random finite level-$\CC$ $\Gamma$-group \cite[Theorem 1.2]{Saw20} and computed the limiting distribution of the random group $\Gal(K^{\#}/K)^{\CC}$ \cite[Theorem 1.1]{Saw20} using \cite[Theorem 1.4]{LWZ19}. We refer \cite[Section 2.4]{Woo22} for a detailed exposition on these works.

Let $\GG_{\Gamma}$ be the set of isomorphism classes of $\Gamma$-groups with finitely many surjections to any finite group, and all continuous finite quotients of order relatively prime to $\left| \Gamma \right|$. (It is same as the definition of $\GG$ in \cite[Section 2.4]{Woo22}.) 
When $\Gamma$ is a trivial group, then $\GG_{\Gamma}$ is homeomorphic to $\GG$ by Lemma \ref{lem22h}. Let $\Sur_{\Gamma}$ be the number of $\Gamma$-equivariant surjections between two $\Gamma$-groups. 
The following theorem will be proved in Section \ref{Sub22} using Theorem \ref{thm22g}. It is an $r$-tuple version of \cite[Corollary 2.22]{Woo22}. 

\begin{theorem} \label{thm1i}
Let $Y = (Y_1, \hdots, Y_r)$ and $X_t = (X_{t,1}, \hdots, X_{t, r})$ ($t = 0, 1, \hdots$) be $r$-tuples of random groups in $\GG_{\Gamma}$. Assume that for every finite $\Gamma$-groups $H_1, \hdots, H_r$, we have
\begin{equation} \label{eq1d}
\lim_{t \rightarrow \infty} \EE(\prod_{s=1}^{r} \# \Sur_{\Gamma}(X_{t, s}, H_s)) = \EE(\prod_{s=1}^{r} \# \Sur_{\Gamma}(Y_{s}, H_s)) = O \left ( ( \prod_{s=1}^{r} \left| H_s \right| )^{O(1)} \right ).  
\end{equation}
Then the distributions of $X_t \in \GG_{\Gamma}^r$ weakly converge in distribution to $Y$ as $t \rightarrow \infty$.
\end{theorem}

Motivated by the non-abelian Cohen-Lenstra heuristics (i.e. the study of the distribution of the random profinite group $\Gal(K^{\#}/K)$), Sawin and Wood \cite{SW22b} studied the moment problem for random objects in a general category. They proved existence, uniqueness and robustness for the moment problem on a \textit{diamond category} \cite[Theorem 1.6]{SW22b}. Moreover, unlike the previous approaches, their method provides a way to construct a distribution from its moments. 


In Section \ref{Sec4}, we extend Theorem \ref{thm1h} to the joint distribution of two random profinite groups in $\GG$. We prove the following non-abelian analogue of Theorem \ref{thm1e} using Theorem \ref{thm1i}. 

\begin{theorem} \label{thm1j}
(Theorem \ref{thm4d}) Let $u \geq 0$ be an integer, $r_1, \hdots, r_{n+u}$ be independent uniform random elements of $F_n$ and $b_{n, 1}, \hdots, b_{n, n+u}$ be given elements of $F_n$ for each $n$. Assume that $\lim_{n \rightarrow \infty} d_n = \infty$, where $d_n$ is the maximum size of a subset $S \subset \left< b_{n, 1}, \hdots, b_{n, n+u} \right>$ which can be extended to a generating set of $F_n$. Then the joint distributions of
$$
(F_n/\left< r_1, \hdots, r_{n+u} \right>, \, F_n/\left< r_1 b_{n, 1}, \hdots, r_{n+u} b_{n, n+u} \right>)
$$
weakly converge in distribution to the probability measure $\mu_u \times \mu_u$ on $\GG^2$ as $n \rightarrow \infty$.
\end{theorem}


Recently, Sawin and Wood \cite{SW22a} computed the limiting distribution of the profinite completion of the fundamental group of a random $3$-manifold as the genus goes to infinity. They used random $3$-manifolds given by the Dunfield-Thurston model of random Heegaard splittings \cite{DT06}. This result, along with the works on the distribution of the Galois group $\Gal(K^{\#}/K)$ (e.g. \cite{LWZ19}, \cite{Saw20}) indicates that there would be more applications of the probability theory of random groups. In the future, we hope to find a new application of the probability theory of random groups to number theory.

\section{Mixed moments determine the joint distribution} \label{Sec2}

\subsection{Joint distribution of finite abelian groups} \label{Sub21}

In this section, we prove Theorem \ref{thm21c}. As mentioned in Remark \ref{rmk1new1}, this result is not essentially new. We include the proof for the completeness. 
Before doing that, we present a simple proof for the following special case of Theorem \ref{thm21c}. Recall that $\GG_p$ is the set of isomorphism classes of finite abelian $p$-groups and $c_{\infty}(p) := \prod_{k=1}^{\infty} (1-p^{-k})$.

\begin{proposition} \label{prop21a}
Let $r \geq 1$ be an integer, $p$ be a prime such that $2^{\frac{1}{r}} c_{\infty}(p) > 1$ and $\nu$ be a probability measure on $\GG_p^r$. If
\begin{equation} \label{eq21a}
\sum_{(G_1, \hdots, G_r) \in \GG_p^r} \nu(G_1, \hdots, G_r) \prod_{i=1}^{r} \# \Sur(G_i, H_i) = 1
\end{equation}
for every $H_1, \hdots, H_r \in \GG_p$, then we have $\nu(G_1, \hdots, G_r) = \prod_{i=1}^{r} \frac{c_{\infty}(p)}{\left| \Aut(G_i) \right|}$.
\end{proposition}

\begin{proof}
We follow the proof of \cite[Lemma 8.2]{EVW16}. Let $\beta = c_{\infty}(p)^{-r}-1 \in (0, 1)$ and
$$
\alpha(G_1, \hdots, G_r) = \nu(G_1, \hdots, G_r) \prod_{i=1}^{r} \left| \Aut(G_i) \right| .
$$
\begin{itemize}
    \item It is clear by the equation (\ref{eq21a}) that $\alpha(H_1, \hdots, H_r) \leq 1$ for every $H_1, \hdots, H_r \in \GG_p$. 
    
    \item By the equation (\ref{eq21a}), we have
    \begin{align*}
    1 & = \alpha(H_1, \hdots, H_r) + \sum_{\substack{(G_1, \hdots, G_r) \\ \neq (H_1, \hdots, H_r)}} \alpha(G_1, \hdots, G_r) \prod_{i=1}^{r} \frac{\left| \Sur(G_i, H_i) \right|}{\left| \Aut(G_i) \right| } \\ 
    & \leq \alpha(H_1, \hdots, H_r) - 1 + \sum_{(G_1, \hdots, G_r) \in \GG_p^r} \prod_{i=1}^{r} \frac{\left| \Sur(G_i, H_i) \right|}{\left| \Aut(G_i) \right| } \\
    & = \alpha(H_1, \hdots, H_r) -1 + \frac{1}{c_{\infty}(p)^r}
    \end{align*}    
    so $\alpha(H_1, \hdots, H_r) \geq 1 - \beta$.

    \item Applying the lower bound $\alpha(G_1, \hdots, G_r) \geq 1 - \beta$ (for all $(G_1, \cdots, G_r) \neq (H_1, \cdots, H_r)$) to the equation (\ref{eq21a}), we obtain $\alpha(H_1, \hdots, H_r) \leq 1 - \beta + \beta^2$. 

    \item Applying the upper bound $\alpha(G_1, \hdots, G_r) \leq 1 - \beta + \beta^2$ (for all $(G_1, \cdots, G_r) \neq (H_1, \cdots, H_r)$) to the equation (\ref{eq21a}), we obtain $\alpha(H_1, \hdots, H_r) \geq 1 - \beta + \beta^2 - \beta^3$. 
\end{itemize}
Iterating this procedure, we conclude that $\alpha(H_1, \hdots, H_r) = \sum_{k=0}^{\infty}(-\beta)^k = c_{\infty}(p)^r$.
\end{proof}

The next lemma is a generalization of \cite[Theorem 8.2]{Woo17}. Its proof is virtually the same as the original theorem, so we omit some details.

\begin{lemma} \label{lem21b}
Let $r, s \geq 1$ be integers, $p_1, \hdots, p_s$ be distinct primes and $m_1, \hdots, m_s \geq 1$ be integers. Let $M_j$ be the set of partitions with at most $m_j$ parts and $M = M_1 \times \cdots \times M_s$. For $\mu \in M$, let $\mu^j = (\mu^j_1 \geq \cdots \geq \mu^j_{m_j}) \in M_j$ be its $j$th entry. For each $\mu = (\mu(1), \hdots, \mu(r)) \in M^r$, let $x_{\mu}$ and $y_{\mu}$ be non-negative real numbers. For each $\lambda = (\lambda(1), \hdots, \lambda(r)) \in M^r$, let $C_{\lambda}$ be a non-negative real number such that
$$
C_{\lambda} \leq \prod_{k=1}^{r} \prod_{j=1}^{s} F^{m_j} p_j^{ \sum_{i=1}^{m_j} \frac{(\lambda(k)_i^j)^2}{2} }
$$
for some constant $F > 0$. Suppose that for all $\lambda \in M^r$,
\begin{equation*} 
\sum_{\mu \in M^r} x_{\mu} \prod_{k=1}^{r}\prod_{j=1}^{s} p_j^{\sum_{i=1}^{m_j} \lambda(k)_i^j \mu(k)_i^j}
=  \sum_{\mu \in M^r} y_{\mu} \prod_{k=1}^{r} \prod_{j=1}^{s} p_j^{\sum_{i=1}^{m_j} \lambda(k)_i^j \mu(k)_i^j}
= C_{\lambda}.
\end{equation*}
Then we have $x_{\mu} = y_{\mu}$ for all $\mu \in M^r$.
\end{lemma}

\begin{proof}
We take the total order $\leq$ on $M$ as in the proof of \cite[Theorem 8.2]{Woo17} and define the partial order $\leq_r$ on $M^r$ by $\lambda \leq_r \lambda'$ if and only if $\lambda(k) \leq \lambda'(k)$ for each $k$. 
For each $1 \leq j \leq s$, $1 \leq k \leq r$ and $\nu \in M^r$, let
$$
H_{m_j, p_j, \nu(k)^j}(z) = \sum_{\substack{d_1, \hdots, d_{m_j} \geq 0 \\ d_2 + \cdots + d_{m_j} \leq \nu(k)^j_1}} a_{d_1, \hdots, d_{m_j}}^{j, k} z_1^{d_1} \cdots z_{m_j}^{d_{m_j}}
$$
be the function given in \cite[Lemma 8.1]{Woo17}. For $\lambda \in M^r$, define
$$
A_{\lambda} = \prod_{k=1}^{r} \prod_{j=1}^{s} a^{j, k}_{\lambda(k)^{j}_{1} - \lambda(k)^{j}_{2}, \, \lambda(k)^{j}_{2} - \lambda(k)^{j}_{3}, \hdots, \, \lambda(k)^{j}_{m_j}}.
$$
Then the sum $\sum_{\lambda \in M^r} A_{\lambda} C_{\lambda}$ converges absolutely and
$$
\sum_{\lambda \in M^r} A_{\lambda} C_{\lambda}
= x_{\nu} u_{\nu} + \sum_{\mu \in M^r, \, \mu <_r \nu} x_{\mu} \sum_{\lambda \in M^r} A_{\lambda} \prod_{k=1}^{r} \prod_{j=1}^{s} p_j^{\sum_{i=1}^{m_j} \lambda(k)^j_i \mu(k)^j_i}
$$
for
$$
u_{\nu} = \prod_{k=1}^{r} \prod_{j=1}^{s} H_{m_j, p_j, \nu(k)^j}(p_j^{\nu(k)^j_1}, p_j^{\nu(k)^j_1 + \nu(k)^j_2}, \hdots, p_j^{\nu(k)^j_1 + \cdots + \nu(k)^j_{m_j}}) \neq 0.
$$
Therefore $x_{\nu}$ is determined by the numbers $\left\{ C_{\lambda} \right\}_{\lambda \in M^r}$ and $\left\{ x_{\pi} \right\}_{\pi <_r \nu}$.
\end{proof}

Let $p_1, \hdots, p_s$ be distinct primes, $G_j$ be a finite abelian $p_j$-group of type $\lambda_{j}$ for each $j$ and $G = G_1 \times \cdots \times G_s$. For a partition $\lambda$, let $\lambda'$ be the conjugate of $\lambda$. We define $m(G_j) :=  p_j^{\sum_{i} \frac{(\lambda_{j}')_i^2}{2}}$ and $m(G) := \prod_{j=1}^{s} m(G_j)$.
Now we prove the main result of this section. The proof is largely based on the proof of \cite[Theorem 8.3]{Woo17} so we omit some details as before.

\begin{theorem} \label{thm21c}
Let $a$ be a positive integer and $\calA$ be the set of finite abelian groups with exponent dividing $a$. Let $Y = (Y_1, \hdots, Y_r)$, $X_n = (X_{n, 1}, \hdots, X_{n, r})$ ($n = 1, 2, \hdots$) be $r$-tuples of random finitely generated abelian groups. Suppose that for every $G_1, \hdots, G_r \in \calA$, we have
\begin{equation} \label{eq21c}
\lim_{n \rightarrow \infty} \EE(\prod_{k=1}^{r} \# \Sur(X_{n, k}, G_k))
= \EE(\prod_{k=1}^{r} \# \Sur(Y_k, G_k))
= O \left ( \prod_{k=1}^{r} m(G_k) \right ).
\end{equation}
Then for every $H_1, \hdots, H_r \in \calA$, 
\begin{equation} \label{eq21d}
\lim_{n \rightarrow \infty} \PP \begin{pmatrix}
X_{n, k} \otimes \Z / a \Z \cong H_k \\
\text{for } 1 \leq k \leq r
\end{pmatrix}
= \PP \begin{pmatrix}
Y_k \otimes \Z / a \Z \cong H_k \\
\text{for } 1 \leq k \leq r
\end{pmatrix}.   
\end{equation}
\end{theorem}

\begin{proof}
Assume that the limit $\lim_{n \rightarrow \infty} \PP (X_{n, k} \otimes \Z /a \Z \cong H_k \text{ for } 1 \leq k \leq r )$ converges for every $H_1, \hdots, H_r \in \calA$. Let $a = \prod_{j=1}^{s} p_j^{m_j}$ be the prime factorization of $a$. For each $G_k \in \calA$, we can choose $G'_k \in \calA$ such that the sum
$$
\sum_{B_k \in \calA} \frac{\# \Hom(B_k, G_k)}{\# \Hom(B_k, G'_k)}
$$
converges. (This can be found in the proof of \cite[Theorem 8.3]{Woo17}.) For all $n \geq 1$ and $H_1, \hdots, H_r \in \calA$, we have
\begin{align*}
& \PP(\forall k \;\; X_{n,k} \otimes \Z / a \Z \cong H_k) \prod_{k=1}^{r} \# \Hom(H_k, G'_k) \\
\leq \, & \sum_{B_1, \hdots, B_r \in \calA} \PP(\forall k \;\; X_{n,k} \otimes \Z / a \Z \cong B_k) \prod_{k=1}^{r} \# \Hom(B_k, G'_k) \\
= \, & \EE ( \prod_{k=1}^{r} \# \Hom(X_{n,k}, G'_k) ) \\
= \, & \sum_{K_1 \leq G_1'} \cdots \sum_{K_r \leq G_r'} \EE ( \prod_{k=1}^{r} \# \Sur(X_{n,k}, K_k) ) \\
\leq \, & D_{G_1, \hdots, G_r}
\end{align*}
for some constant $D_{G_1, \hdots, G_r}$ by the equation (\ref{eq21c}). Then,
\begin{equation*}
\PP(\forall k \;\; X_{n,k} \otimes \Z / a \Z \cong H_k) \prod_{k=1}^{r} \# \Hom(H_k, G_k) \leq D_{G_1, \hdots, G_r} \prod_{k=1}^{r} \frac{\# \Hom(H_k, G_k)}{\# \Hom(H_k, G'_k)}
\end{equation*}
and the sum
$$
\sum_{H_1, \hdots, H_r \in \calA} D_{G_1, \hdots, G_r} \prod_{k=1}^{r} \frac{\# \Hom(H_k, G_k)}{\# \Hom(H_k, G'_k)}
= D_{G_1, \hdots, G_r} \prod_{k=1}^{r} \left ( \sum_{H_k \in \calA} \frac{\# \Hom(H_k, G_k)}{\# \Hom(H_k, G'_k)} \right )
$$
converges. Thus we have
\begin{equation*}
\begin{split}
& \sum_{H_1, \hdots, H_r \in \calA} 
\lim_{n \rightarrow \infty} \PP(\forall k \;\; X_{n,k} \otimes \Z / a \Z \cong H_k) \prod_{k=1}^{r} \# \Hom(H_k, G_k) \\
= \, & \lim_{n \rightarrow \infty} \sum_{H_1, \hdots, H_r \in \calA} \PP(\forall k \;\; X_{n,k} \otimes \Z / a \Z \cong H_k) \prod_{k=1}^{r} \# \Hom(H_k, G_k) \\
= \, & \lim_{n \rightarrow \infty} \EE ( \prod_{k=1}^{r} \# \Hom(X_{n,k}, G_k) )
\end{split}    
\end{equation*}
by the Lebesgue dominated convergence theorem. Now the equation (\ref{eq21c}) implies that
\begin{equation*}
\begin{split}
& \lim_{n \rightarrow \infty} \EE ( \prod_{k=1}^{r} \# \Hom(X_{n,k}, G_k) ) \\
= \, & \sum_{\WG_1 \leq G_1} \cdots \sum_{\WG_r \leq G_r} \lim_{n \rightarrow \infty}  \EE ( \prod_{k=1}^{r} \# \Sur(X_{n,k}, \WG_k) ) \\
= \, & O \left ( \prod_{k=1}^{r} \left ( \sum_{\WG_k \leq G_k} m(\WG_k) \right ) \right ).
\end{split}    
\end{equation*}
Let $G_{k, j}$ be the $p_j$-Sylow subgroup of $G_k$ and $\lambda(k)^{j}$ be the conjugate of the type of $G_{k, j}$. 
Since the exponent of $G_{k, j}$ divides $p_j^{m_j}$, $\lambda(k)^{j}$ has at most $m_j$ parts. By \cite[Lemma 3.1]{Lee23b}, there is a constant $F>0$ such that
$$
\sum_{\WG_k \leq G_k} m(\WG_k) 
= \prod_{j=1}^{s} \left ( \sum_{\WG_{k, j} \leq G_{k, j}} m(\WG_{k, j}) \right )
\leq F^{\sum_{j=1}^{s} m_j} m(G_k)
$$
for each $k$. This implies that
$$
\prod_{k=1}^{r} \left ( \sum_{\WG_k \leq G_k} m(\WG_k) \right )
\leq \prod_{k=1}^{r} F^{ \sum_{j=1}^{s} m_j} m(G_k)
= \prod_{k=1}^{r} \prod_{j=1}^{s} F^{m_j} p_j^{\sum_{i} \frac{ (\lambda(k)^j_i)^2 }{2}}.
$$
Now let $M$ be as in Lemma \ref{lem21b}. By enlarging the constant $F$ if necessary, we obtain
$$
C_{\lambda} = \lim_{n \rightarrow \infty} \EE ( \prod_{k=1}^{r} \# \Hom(X_{n,k}, G_k) )
\leq \prod_{k=1}^{r} \prod_{j=1}^{s} F^{m_j} p_j^{\sum_{i} \frac{ (\lambda(k)^j_i)^2 }{2}}
$$
where $\lambda = (\lambda(1), \hdots, \lambda(r)) \in M^r$ and the type of $G_{k, j}$ is the conjugate of $\lambda(k)^j \in M_j$ for every $1 \leq k \leq r$ and $1 \leq j \leq s$. For $\mu = (\mu(1), \hdots, \mu(r)) \in M^r$ and $H_1, \hdots, H_r \in \calA$ such that the type of the $p_j$-Sylow subgroup of $H_k$ is the conjugate of $\mu(k)^j \in M_j$, define
\begin{equation*}
\begin{split}
x_{\mu} & = \lim_{n \rightarrow \infty} \PP (X_{n, k} \otimes \Z /a \Z \cong H_k \text{ for } 1 \leq k \leq r ), \\
y_{\mu} & = \PP (Y_k \otimes \Z /a \Z \cong H_k \text{ for } 1 \leq k \leq r ).
\end{split}    
\end{equation*}
Then the equation (\ref{eq21d}) follows from Lemma \ref{lem21b}. The existence of the limit $\lim_{n \rightarrow \infty} \PP (X_{n, k} \otimes \Z /a \Z \cong H_k \text{ for } 1 \leq k \leq r )$ can be proved as in \cite[Theorem 8.3]{Woo17}.
\end{proof}

\subsection{Joint distribution of finite $\Gamma$-groups} \label{Sub22}

Let $\Gamma$ be a finite group, $H_1, \hdots, H_r$ be $\Gamma$-groups and $H_i' = H_i \rtimes \Gamma$ for each $i$. Define $[H_i']$-groups and $[H_i']$-homomorphisms as in \cite[Definition 2.1 and 2.2]{Saw20}. Let $\Sur_{[H_i']}(G_1, G_2)$ be the set of surjective $[H_i']$-homomorphisms from $G_1$ to $G_2$. 
In this section, we prove Theorem \ref{thm1i}. 
First we prove the following two lemmas which generalize \cite[Lemma 2.8 and 2.9]{Saw20}.

\begin{lemma} \label{lem22a}
For each $1 \leq s \leq r$, let $G_{s1}, \hdots, G_{s m_s}$ be finite simple $[H_s']$-groups that are not pairwise $[H_s']$-isomorphic and $j_s$ be a positive integer. Let $\widetilde{\mu}, \, \widetilde{\mu_0}, \, \widetilde{\mu_1}, \hdots$ be measures on $\GG_1 \times \cdots \times \GG_r$ where $\GG_s$ is the set of isomorphism classes of $[H_s']$-groups of the form $\prod_{i=1}^{m_s} G_{si}^{e_{si}}$. 
Assume that $j_1 \leq m_1$ and for every $k_{s i_s} \in \Z_{\geq 0}$ ($1 \leq s \leq r$, $j_s \leq i_s \leq m_s$), we have
\begin{equation} \label{eq22c}
\sum_{\substack{e_{s i_s} \in \Z_{\geq 0} \text{ for} \\ 1 \leq s \leq r, \, j_s \leq i_s \leq m_s}}
S_{j_1, \hdots, j_r}(\bfe, \bfk) \widetilde{\mu} (G_{j_1, \hdots, j_r}^{\bfe}) = O(O(1)^{\sum_{s=1}^{r} \sum_{i_s=j_s}^{m_s} k_{s i_s}})
\end{equation}
and
\begin{equation} \label{eq22d}
\begin{split}
& \lim_{t \rightarrow \infty} \sum_{\substack{e_{s i_s} \in \Z_{\geq 0} \text{ for} \\ 1 \leq s \leq r, \, j_s \leq i_s \leq m_s}}  S_{j_1, \hdots, j_r}(\bfe, \bfk) \widetilde{\mu_t} (G_{j_1, \hdots, j_r}^{\bfe}) \\
= \, & \sum_{\substack{e_{s i_s} \in \Z_{\geq 0} \text{ for} \\ 1 \leq s \leq r, \, j_s \leq i_s \leq m_s}}  S_{j_1, \hdots, j_r}(\bfe, \bfk) \widetilde{\mu} (G_{j_1, \hdots, j_r}^{\bfe}),
\end{split}
\end{equation}
where
$$
S_{j_1, \hdots, j_r}(\bfe, \bfk) := \prod_{s=1}^{r} \# \Sur_{[H_s']}\left ( \prod_{i_s=j_s}^{m_s}G_{s i_s}^{e_{s i_s}}, \prod_{i_s=j_s}^{m_s}G_{s i_s}^{k_{s i_s}} \right )
$$
and
$$
G_{j_1, \hdots, j_r}^{\bfe} := \prod_{i_1=j_1}^{m_1} G_{1i_1}^{e_{1i_1}} \times \cdots \times \prod_{i_r=j_r}^{m_r} G_{ri_r}^{e_{ri_r}}.
$$
Then for every $k_{si_s} \in \Z_{\geq 0}$ ($1 \leq s \leq r$, $j_s' \leq i_s \leq m_s$) with 
$$
(j_1', \hdots, j_r') = (j_1+1, j_2, \hdots, j_r),
$$
we have
\begin{equation*}
\begin{split}
& \lim_{t \rightarrow \infty} 
\sum_{\substack{e_{s i_s} \in \Z_{\geq 0} \text{ for} \\ 1 \leq s \leq r, \, j_s' \leq i_s \leq m_s}} 
S_{j_1', \hdots, j_r'}(\bfe, \bfk) \widetilde{\mu_t} (G_{j_1', \hdots, j_r'}^{\bfe}) \\
= \, & \sum_{\substack{e_{s i_s} \in \Z_{\geq 0} \text{ for} \\ 1 \leq s \leq r, \, j_s' \leq i_s \leq m_s}} 
S_{j_1', \hdots, j_r'}(\bfe, \bfk) \widetilde{\mu} (G_{j_1', \hdots, j_r'}^{\bfe}).
\end{split}
\end{equation*}
\end{lemma}

\begin{proof}
Fix $k_{s i_s} \in \Z_{\geq 0}$ ($1 \leq s \leq r$, $j_s' \leq i_s \leq m_s$). For $e' \in \Z_{\geq 0}$, define
\begin{equation*}
\begin{split}
m(e') & = \sum_{\substack{e_{1 j_1} = e' \text{ and } e_{s i_s} \in \Z_{\geq 0} \\ \text{for }  1 \leq s \leq r, \, j_s' \leq i_s \leq m_s}}
S_{j_1', \hdots, j_r'}(\bfe, \bfk) \widetilde{\mu} (G_{j_1, \hdots, j_r}^{\bfe}), \\
m_t(e') & = \sum_{\substack{e_{1 j_1} = e' \text{ and } e_{s i_s} \in \Z_{\geq 0} \\ \text{for }  1 \leq s \leq r, \, j_s' \leq i_s \leq m_s}}
S_{j_1', \hdots, j_r'}(\bfe, \bfk) \widetilde{\mu_t} (G_{j_1, \hdots, j_r}^{\bfe}).
\end{split}    
\end{equation*}
Then for any $k_{1 j_1} \in \Z_{\geq 0}$, we have
\begin{equation*}
\begin{split}
& \sum_{e'=0}^{\infty} \# \Sur_{[H_1']} (G_{1 j_1}^{e'}, G_{1 j_1}^{k_{1 j_1}} )  m(e') \\
= \, & \sum_{\substack{e_{s i_s} \in \Z_{\geq 0} \text{ for} \\ 1 \leq s \leq r, \, j_s \leq i_s \leq m_s}}
S_{j_1, \hdots, j_r}(\bfe, \bfk) \widetilde{\mu} (G_{j_1, \hdots, j_r}^{\bfe}) \\
= \, & O(O(1)^{k_{1 j_1}})
\end{split}
\end{equation*}
by \cite[Lemma 2.5]{Saw20} and the equation (\ref{eq22c}). (Note that for fixed $k_{si_s}$ ($1 \leq s \leq r$, $j_s' \leq i_s \leq m_s$), $O(O(1)^{\sum_{s=1}^{r} \sum_{i_s=j_s}^{m_s} k_{s i_s}}) = O(O(1)^{k_{1 j_1}})$.) The same formula holds for $m_t$ and $\widetilde{\mu_t}$ for each $t$, so the equation (\ref{eq22d}) implies that
\small
$$
\lim_{t \rightarrow \infty} \sum_{e'=0}^{\infty} \# \Sur_{[H_1']} (G_{1 j_1}^{e'}, G_{1 j_1}^{k_{1 j_1}} )  m_t(e')
= \sum_{e'=0}^{\infty} \# \Sur_{[H_1']} (G_{1 j_1}^{e'}, G_{1 j_1}^{k_{1 j_1}} )  m(e')
= O(O(1)^{k_{1 j_1}}).
$$
\normalsize
Now \cite[Lemma 2.7]{Saw20} implies that $\lim_{t \rightarrow \infty} m_t(0) = m(0)$. This finishes the proof.
\end{proof}

\begin{lemma} \label{lem22b}
Let $H_s$, $G_{s i_s}$, $\widetilde{\mu}$ and $\widetilde{\mu_t}$ be as in Lemma \ref{lem22a}. Assume that the equations (\ref{eq22c}) and (\ref{eq22d}) hold for $j_1 = \hdots = j_r = 1$. Then we have
$$
\lim_{t \rightarrow \infty} \widetilde{\mu_t}(1, \hdots, 1) = \widetilde{\mu}(1, \hdots, 1).
$$
\end{lemma}

\begin{proof}
By a repeated use of Lemma \ref{lem22a}, one can show that the equation (\ref{eq22d}) holds for $(j_1, \hdots, j_r) = (m_1+1, \hdots, m_r+1)$. 
\end{proof}

For an $[H_s']$-group $T$, let $Q_s(T)$ be the quotient of $T$ by the intersection of its maximal proper $H_s'$-invariant normal subgroups equipped with a natural $[H_s']$-group structure. 
Let $\CC_1, \hdots, \CC_r$ be finite sets of finite $\Gamma$-groups such that $(\left| \CC_s \right|, \left| \Gamma \right|) = 1$ for each $s$.

\begin{lemma} \label{lem22c}
(\cite[Lemma 2.12]{Saw20}) There exist finitely many finite simple $[H_s']$-groups $G_i$ (up to isomorphism) such that there exists an extension $1 \rightarrow G_i \rightarrow G \rightarrow H_s \rightarrow 1$ of $\Gamma$-groups compatible with the actions of $H_s$ and $\Gamma$ on $G_i$ by outer automorphisms, where $G$ is a level-$\CC_s$ $\Gamma$-group.
\end{lemma}

For each $1 \leq s \leq r$, let $G_{s1}, \hdots, G_{s m_s}$ be pairwise non-isomorphic representatives of the isomorphism classes of $G_i$ discussed in Lemma \ref{lem22c}. For a finite level-$\CC_s$ $\Gamma$-group $G$ and $\pi \in \Hom(G, H_s)$, we have an isomorphism of $[H_s']$-groups
\begin{equation*}
Q_s(\ker \pi) \cong \prod_{i_s=1}^{m_s} G_{s i_s}^{e_{s i_s}}
\end{equation*}
for some $e_{s 1}, \hdots, e_{s m_s} \in \Z_{\geq 0}$ by \cite[Lemma 2.14]{Saw20}. Let $\mfm_s$ be the set of isomorphism classes of finite level-$\CC_s$ $\Gamma$-groups and $\mfm = \mfm_1 \times \cdots \times \mfm_r$. Now we can define a localized measure $\mu^{H_1, \hdots, H_r}$ for each measure $\mu$ on $\mfm$.

\begin{definition} \label{def22d}
Let $\mu$ be a measure on $\mfm$ and $(H_1, \hdots, H_r) \in \mfm$. Define a measure $\mu^{\mathbf{H}} := \mu^{H_1, \hdots, H_r}$ on the set of isomorphism classes of $r$-tuples of the form
$$
\prod_{i_1=1}^{m_1} G_{1 i_1}^{e_{1 i_1}} \times \cdots \times \prod_{i_r=1}^{m_r} G_{r i_r}^{e_{r i_r}}
$$
by
\begin{equation*}
\begin{split}
& \mu^{\mathbf{H}}(E_1, \hdots, E_r) \\
:= & \, \int \prod_{s=1}^{r} \left| \left\{ \pi_s \in \Sur_{\Gamma}(X_s, H_s) \mid Q_s(\ker \pi_s) \cong E_s \right\} \right| d \mu (X_1, \hdots, X_r).
\end{split}    
\end{equation*}
\end{definition}

The following two lemmas extend \cite[Lemma 2.16 and 2.19]{Saw20}.

\begin{lemma} \label{lem22e}
$$
\mu^{\mathbf{H}}(1, \hdots, 1) = \left ( \prod_{s=1}^{r} \# \Aut(H_s) \right ) \mu(H_1, \hdots, H_r).
$$
\end{lemma}

\begin{lemma} \label{lem22f}
Let $\mu$ be a measure on $\mfm$ and $F_s = \prod_{i_s=1}^{m_s} G_{s i_s}^{k_{s i_s}}$ for each $s$. Then we have
\begin{align*}
& \int \prod_{s=1}^{r} \# \Sur_{[H_s']}(E_s, F_s) d\mu^{\mathbf{H}}(E_1, \hdots, E_r) \\
= \, & \sum_{\substack{1 \rightarrow F_s \rightarrow G_s \rightarrow H_s \rightarrow 1 \\ \text{for } 1 \leq s \leq r}}  \int \prod_{s=1}^{r} \frac{\# \Sur_{\Gamma}(X_s, G_s)}{\# \Aut_{F_s, H_s}(G_s)} d\mu(X_1, \hdots, X_r),
\end{align*}
where the sum is over exact sequences of $\Gamma$-groups compatible with the actions of $H_s$ and $\Gamma$ on $F_s$ by outer automorphisms for each $1 \leq s \leq r$.
\end{lemma}

\begin{proof}
By \cite[Lemma 2.17 and 2.18]{Saw20}, we have
\begin{align*}
& \int \prod_{s=1}^{r} \# \Sur_{[H_s']}(E_s, F_s) d\mu^{\mathbf{H}}(E_1, \hdots, E_r) \\
= \, & \int \prod_{s=1}^{r} \left ( \sum_{\pi_s \in \Sur(X_s, H_s)} \# \Sur_{[H_s']}(Q_s(\ker \pi_s), F_s) \right ) d \mu(X_1, \hdots, X_r) \\
= \, & \int \prod_{s=1}^{r} \left ( \sum_{1 \rightarrow F_s \rightarrow G_s \rightarrow H_s \rightarrow 1} \frac{\# \Sur_{\Gamma}(X_s, G_s)}{\# \Aut_{F_s, H_s}(G_s)} \right ) d \mu(X_1, \hdots, X_r)  \\
= \, & \sum_{\substack{1 \rightarrow F_s \rightarrow G_s \rightarrow H_s \rightarrow 1 \\ \text{for } 1 \leq s \leq r}}  \int \prod_{s=1}^{r} \frac{\# \Sur_{\Gamma}(X_s, G_s)}{\# \Aut_{F_s, H_s}(G_s)} d\mu(X_1, \hdots, X_r). 
\end{align*}
\end{proof}

The following theorem is a generalization of \cite[Theorem 1.2]{Saw20}.

\begin{theorem} \label{thm22g}
Let $\mu, \, \mu_0, \, \mu_1, \hdots$ be measures on $\mfm$. Assume that for every $(G_1, \hdots, G_r) \in \mfm$, we have
\begin{equation} \label{eq22f}
\begin{split}
& \lim_{t \rightarrow \infty} \int \prod_{s=1}^{r} \# \Sur_{\Gamma} (X_s, G_s) d \mu_t(X_1, \hdots, X_r) \\
= & \, \int \prod_{s=1}^{r} \# \Sur_{\Gamma} (X_s, G_s) d \mu(X_1, \hdots, X_r)
\end{split}
\end{equation}
and
\begin{equation} \label{eq22g}
\int \prod_{s=1}^{r} \# \Sur_{\Gamma} (X_s, G_s) d \mu(X_1, \hdots, X_r) = O \left ( ( \prod_{s=1}^{r} \left| G_s \right| )^{O(1)} \right ).
\end{equation}
Then for all $(H_1, \hdots, H_r) \in \mfm$, we have
\begin{equation} \label{eq22h}
\lim_{t \rightarrow \infty} \mu_t(H_1, \hdots, H_r) 
= \mu(H_1, \hdots, H_r).
\end{equation}
\end{theorem}

\begin{proof}
Fix $(H_1, \hdots, H_r) \in \mfm$ and let $F_s = \prod_{i_s=1}^{m_s} G_{s i_s}^{k_{s i_s}}$ for $1 \leq s \leq r$. Then we have
\begin{align*}
& \lim_{t \rightarrow \infty} \int \prod_{s=1}^{r} \# \Sur_{[H_s']}(E_s, F_s) d\mu_t^{\mathbf{H}}(E_1, \hdots, E_r) \\
= \, & \lim_{t \rightarrow \infty} \sum_{\substack{1 \rightarrow F_s \rightarrow G_s \rightarrow H_s \rightarrow 1 \\ \text{for } 1 \leq s \leq r}}  \int \prod_{s=1}^{r} \frac{\# \Sur_{\Gamma}(X_s, G_s)}{\# \Aut_{F_s, H_s}(G_s)} d\mu_t(X_1, \hdots, X_r) \\
= \, & \sum_{\substack{1 \rightarrow F_s \rightarrow G_s \rightarrow H_s \rightarrow 1 \\ \text{for } 1 \leq s \leq r}} \left ( \lim_{t \rightarrow \infty} \int \prod_{s=1}^{r} \frac{\# \Sur_{\Gamma}(X_s, G_s) d \mu_t(X_1, \hdots, X_r)}{\# \Aut_{F_s, H_s}(G_s)} \right ) \\
= \, & \sum_{\substack{1 \rightarrow F_s \rightarrow G_s \rightarrow H_s \rightarrow 1 \\ \text{for } 1 \leq s \leq r}} \left ( \int \prod_{s=1}^{r} \frac{\# \Sur_{\Gamma}(X_s, G_s) d \mu(X_1, \hdots, X_r)}{\# \Aut_{F_s, H_s}(G_s)} \right ) \\
= \, & \int \prod_{s=1}^{r} \# \Sur_{[H_s']}(E_s, F_s) d\mu^{\mathbf{H}}(E_1, \hdots, E_r),
\end{align*}
where the first and last equalities are due to Lemma \ref{lem22f}, the second equality is due to \cite[Lemma 2.18]{Saw20} and the third equality follows from Assumption (\ref{eq22f}).
By Lemma \ref{lem22f} and the assumption (\ref{eq22g}), we have
\begin{equation*}
\begin{split}
& \int \prod_{s=1}^{r} \# \Sur_{[H_s']}(E_s, F_s) d\mu^{\mathbf{H}}(E_1, \hdots, E_r) \\
= \, & \sum_{\substack{1 \rightarrow F_s \rightarrow G_s \rightarrow H_s \rightarrow 1 \\ \text{for } 1 \leq s \leq r}}  \int \prod_{s=1}^{r} \frac{\# \Sur_{\Gamma}(X_s, G_s)}{\# \Aut_{F_s, H_s}(G_s)} d\mu(X_1, \hdots, X_r) \\
= \, & \prod_{s=1}^{r} \left ( \sum_{1 \rightarrow F_s \rightarrow G_s \rightarrow H_s \rightarrow 1} O( \left| G_s \right|)^{O(1)} \right ).
\end{split}    
\end{equation*}
Since $G_{s1}, \hdots, G_{s m_s}$ are determined by $H_s$, we have 
$$
\left| G_s \right| = \left| F_s \right| \left| H_s \right| = O(O(1)^{\sum_{i_s=1}^{m_s} k_{s i_s}})
$$
for a fixed $H_s$. By \cite[Lemma 2.18]{Saw20}, the number of exact sequences $1 \rightarrow F_s \rightarrow G_s \rightarrow H_s \rightarrow 1$ (compatible with the actions) is $O(O(1)^{\sum_{i_s=1}^{m_s} k_{s i_s}})$. Thus we have
$$
\int \prod_{s=1}^{r} \# \Sur_{[H_s']}(E_s, F_s) d\mu^{\mathbf{H}}(E_1, \hdots, E_r) = O(O(1)^{\sum_{s=1}^{r} \sum_{i_s=1}^{m_s} k_{s i_s} }).
$$
Now we can apply Lemma \ref{lem22b} to $\widetilde{\mu} = \mu^{\mathbf{H}}$ and $\widetilde{\mu_t} = \mu_t^{\mathbf{H}}$. We obtain
$$
\lim_{t \rightarrow \infty} \mu_t^{\mathbf{H}}(1, \hdots, 1) = \mu^{\mathbf{H}}(1, \hdots, 1),
$$
which implies the equation (\ref{eq22h}) by Lemma \ref{lem22e}.
\end{proof}

Recall that $\GG_{\Gamma}$ denotes the set of isomorphism classes of $\Gamma$-groups with finitely many surjections to any finite group, and all continuous finite quotients of order relatively prime to $\left| \Gamma \right|$. 
For a positive integer $m$, let $m'$-$\Gamma$-\textit{group} be a $\Gamma$-group such that every finite quotient has order relatively prime to $m$.
We endow $\GG_{\Gamma}$ with a topology for which a basis of open sets is given by $U_{\CC, H} := \left\{ G \in \GG_{\Gamma} \mid G^{\CC} \cong H \right\}$ for a finite set of finite $\left| \Gamma \right|'$-$\Gamma$-groups $\CC$ and a finite $\left| \Gamma \right|'$-$\Gamma$-group $H$. 
The next lemma shows that if $\Gamma$ is a trivial group, then $\GG_{\Gamma}$ is homeomorphic to $\GG$ defined just before Theorem \ref{thm1h}. 

\begin{lemma} \label{lem22h}
A profinite group $G$ is an element of $\GG$ if and only if $\Sur(G, H)$ is finite for every finite group $H$. 
\end{lemma}

\begin{proof}
Assume that $G \in \GG$. Let $H$ be a finite group of order $\ell$ and $\CC$ be the set of finite groups of order at most $\ell$. Then $G^{\CC}$ is finite, so the set $\Sur(G, H) \cong \Sur(G^{\CC}, H)$ is also finite. Conversely, assume that $\Sur(G, H)$ is finite for every finite group $H$. Then $G$ has finitely many open subgroups of index $n$ for every positive integer $n$. By \cite[Lemma 8.8]{SW22a}, $G^{\CC}$ is finite for every finite set $\CC$ of finite groups. (There are some notational differences between this paper and \cite{SW22a}. The group $G^{\CC}$ in this paper is same as $G^{\DD}$ in the notation of \cite{SW22a}, where $\DD$ is the set of all subgroups of groups in $\CC$. If $\CC$ is a finite set, then $\DD$ is also a finite set so we can apply \cite[Lemma 8.8]{SW22a}.)
\end{proof}

\noindent \textit{Proof of Theorem \ref{thm1i}.} Following the proof of \cite[Theorem 1.1]{LW20}, it is enough to show that
\begin{equation} \label{eq22i}
\lim_{t \rightarrow \infty} \PP (X_{t, s} \in U_s \text{ for } 1 \leq s \leq r) 
= \PP (Y_{s} \in U_s \text{ for } 1 \leq s \leq r)     
\end{equation}
for every basic open sets $U_1, \hdots, U_r \in \GG_{\Gamma}$. Let $U_s = \left\{ G \in \GG_{\Gamma} \mid G^{\CC_s} \cong H_s \right\}$ for a finite set $\CC_s$ of finite $\left| \Gamma \right|'$-$\Gamma$-groups and a finite $\left| \Gamma \right|'$-$\Gamma$-group $H_s$. 
We may assume that $H_s$ is of level-$\CC_s$, because $U_s$ is empty when $H_s$ is not of level-$\CC_s$. 
The equation (\ref{eq22i}) is equivalent to
\begin{equation} \label{eq22j}
\lim_{t \rightarrow \infty} \PP (X_{t, s}^{\CC_s} \cong H_s \text{ for } 1 \leq s \leq r) 
= \PP (Y_s^{\CC_s} \cong H_s \text{ for } 1 \leq s \leq r).
\end{equation}
Let $\mu$ (resp. $\mu_t$) be measures on $\mfm$ given by the probability distribution of $(Y_1^{\CC_1}, \hdots, Y_r^{\CC_r})$ (resp. $(X_{t, 1}^{\CC_1}, \hdots, X_{t, r}^{\CC_r})$). 
For every $(G_1, \hdots, G_r) \in \mfm$, we have $\Sur_{\Gamma}(X_{t, s}, G_s) \cong \Sur_{\Gamma}(X_{t, s}^{\CC_s}, G_s))$ and $\Sur_{\Gamma}(Y_s, G_s) \cong \Sur_{\Gamma}(Y_s^{\CC_s}, G_s)$ so the assumption (\ref{eq1d}) implies that
\begin{equation*}
\lim_{t \rightarrow \infty} \EE(\prod_{s=1}^{r} \# \Sur_{\Gamma}(X_{t, s}^{\CC_s}, G_s)) = \EE(\prod_{s=1}^{r} \# \Sur_{\Gamma}(Y_s^{\CC_s}, G_s)) = O \left ( ( \prod_{s=1}^{r} \left| G_s \right| )^{O(1)} \right ).  
\end{equation*}
Therefore the conditions (\ref{eq22f}) and (\ref{eq22g}) are satisfied and the equation (\ref{eq22j}) follows from Theorem \ref{thm22g}. $\square$

\section{Joint distribution of the cokernels of random $p$-adic matrices} \label{Sec3}

In this section, we provide three applications of Theorem \ref{thm1c} to universality results for the joint distribution of the cokernels of random $p$-adic matrices. The first two of them generalize the main results of \cite{Lee23a} on the asymptotically independent cokernels to non-uniform matrices, and the last one gives a new result on the dependent cokernels. 

\subsection{Independent cokernels} \label{Sub31}

First we consider the joint distribution of $\cok(A_n+t_1I_n), \hdots, \cok(A_n + t_rI_n)$ for given integers $t_1, \hdots, t_r$ and random $p$-adic matrices $A_n$. 

\begin{lemma} \label{lem31a}
Let $t_1, \hdots, t_r$ be integers such that $p \nmid t_j - t_{j'}$ for each $j \neq j'$, $k \geq 1$ be an integer and $H_1, \hdots, H_r \in \GG_p$ be such that $p^kH_j=0$ for all $j$. Also let $R=\Z / p^k \Z$, $V=R^n$, $\left\{ e_1, \hdots, e_n \right\}$ be the standard basis of $V$, $v_1, \hdots, v_n \in V$ and $F_j \in \Sur_R(V, H_j)$ for $1 \leq j \leq r$. Suppose that $F_jv_i = -t_jF_je_i$ for every $1 \leq i \leq n$ and $1 \leq j \leq r$. Then the map
$$
F : V \rightarrow \prod_{j=1}^{r} H_j \;\; (v \mapsto (F_1 v, \hdots, F_r v))
$$
is surjective.
\end{lemma}

\begin{proof}
We proceed by induction on $r$. The case $r=1$ is trivial. Now assume that the statement holds for $r-1 \geq 1$. Define $H_1^c = \prod_{j=2}^{r} H_j$, $F_1^c : V \rightarrow H_1^c$ ($v \mapsto (F_2 v, \hdots, F_r v)$) and $\alpha \in \Aut(H_1^c)$ ($(h_2, \hdots, h_r) \mapsto ((t_1-t_2)h_2, \hdots, (t_1-t_r)h_r)$). Since the map $F_1^c$ is surjective by the induction hypothesis, we have
\begin{equation*}
\begin{split}
& \left< F(v_i+t_1e_i) \mid 1 \leq i \leq n \right> \\
= \, & \left< (0, (t_1-t_2)F_2e_i, \hdots, (t_1-t_r)F_re_i) \mid 1 \leq i \leq n \right> \\
= \, & \left\{ 0 \right\} \times \im(\alpha \circ F_1^c) \\
= \, & \left\{ 0 \right\} \times H_1^c.
\end{split}    
\end{equation*}
This shows that $\left\{ 0 \right\} \times H_1^c \leq \im(F)$, and we also have $(\prod_{j=1}^{r-1} H_j) \times \left\{ 0 \right\} \leq \im(F)$ by the same reason. Thus the map $F$ is surjective.
\end{proof}

\begin{proposition} \label{prop31b}
Let $t_1, \hdots, t_r$ be as in Lemma \ref{lem31a} and $A_n \in \M_n(\Z_p)$ be as in Theorem \ref{thm1a}. Then we have
\begin{equation*}
\lim_{n \rightarrow \infty} \EE(\prod_{j=1}^{r} \# \Sur(\cok(A_n + t_j I_n), H_j)) = 1
\end{equation*}
for every $H_1, \hdots, H_r \in \GG_p$.
\end{proposition}

\begin{proof}
Choose $k \geq 1$ such that $p^k H_j=0$ for all $j$. Let $R = \Z/p^k \Z$ and $X_n \in \M_n(R)$ be the reduction of $A_n$ modulo $p^k$. We understand $X_n$ as an element of $\Hom(W, V)$ for $V=W=R^n$ and let $v_i = X_ne_i$ where $\left\{ e_1, \hdots, e_n \right\}$ is the standard basis of $W$. By Lemma \ref{lem31a}, we have
\begin{equation} \label{eq31a2}
\begin{split}
& \EE(\prod_{j=1}^{r} \# \Sur(\cok(A_n + t_j I_n), H_j)) \\
= \; & \EE(\prod_{j=1}^{r} \# \Sur_R(\cok(X_n + t_j I_n), H_j)) \\
= \; & \sum_{F_1 \in \Sur_R(V, H_1)} \cdots \sum_{F_r \in \Sur_R(V, H_r)} \prod_{i=1}^{n} \PP \begin{pmatrix}
F_jv_i = -t_jF_je_i \\
\text{for } 1 \leq j \leq r
\end{pmatrix} \\
= \; & \sum_{\substack{F = (F_1, \hdots, F_r) \\ \in \Sur_R(V, \, \prod_{j=1}^{r} H_j)}} \prod_{i=1}^{n} 
\PP (Fv_i = -(t_1F_1e_i, \hdots, t_rF_re_i)) \\
= \; & \sum_{F \in \Sur_R(V, H)} \PP(FX_n = U_F)
\end{split}
\end{equation}
where $H = \prod_{j=1}^{r} H_j$ and $U_F \in \Hom(W, H)$ is defined by 
$$
U_Fe_i = -(t_1F_1e_i, \hdots, t_rF_re_i)
$$
for each $F \in \Sur_R(V, H)$. Following the proof of \cite[Theorem 4.12]{NW22a}, we can prove that there are constants $c_2, K_2>0$ (which are independent of $n$) such that
\begin{equation} \label{eq31a3}
\left| \sum_{F \in \Sur_R(V, H)} \PP(FX_n = U_F) - 1 \right| \leq K_2 n^{-c_2}
\end{equation}
for every $n$ and an $\alpha_n$-balanced random matrix $X_n \in \M_n(R)$. (To do this, we need to generalize \cite[Lemma 4.11]{NW22a} to an upper bound of $\PP(FX=A)$ for every $A \in \im (F)$. The proof for $A=0$ works for every $A \in \im (F)$.) We conclude the proof from the equations (\ref{eq31a2}) and (\ref{eq31a3}).
\end{proof}

Since the moments of the Cohen-Lenstra measure are always $1$, the following theorem is an immediate consequence of Theorem \ref{thm1c} and Proposition \ref{prop31b}.

\begin{theorem} \label{thm31c}
Let $t_1, \hdots, t_r$ be as in Lemma \ref{lem31a} and $A_n \in \M_n(\Z_p)$ be as in Theorem \ref{thm1a}. Then we have
\begin{equation*}
\lim_{n \rightarrow \infty} \PP \begin{pmatrix}
\cok(A_n+t_jI_n) \cong H_j \\
\text{for } 1 \leq j \leq r
\end{pmatrix} = \prod_{j=1}^{r} \frac{c_{\infty}(p)}{\left| \Aut(H_j) \right|}
\end{equation*}
for every $H_1, \hdots, H_r \in \GG_p$.
\end{theorem}

Now we consider the joint distribution of $\cok(A_n)$ and $\cok(A_n+B_n)$ for random $p$-adic matrices $A_n$ and a given sequence of $p$-adic matrices $(B_n)_{n \geq 1}$. In this case, it is a bit harder to compute the mixed moments of the cokernels compared to the first case.  

\begin{proposition} \label{prop31d}
Let $u \geq 0$ be an integer, $A_n \in \M_{n \times (n+u)}(\Z_p)$ be as in Theorem \ref{thm1a} and $(B_n)_{n \geq 1}$ be a sequence of matrices where $B_n \in \M_{n \times (n+u)}(\Z_p)$ and $\lim_{n \rightarrow \infty} r_p(B_n) = \infty$ where $r_p(B_n)$ is the rank of $\overline{B_n} \in \M_n(\F_p)$. Then we have
\begin{equation*}
\lim_{n \rightarrow \infty} \EE(\# \Sur(\cok(A_n), H_1) \# \Sur(\cok(A_n + B_n), H_2)) = \frac{1}{\left| H_1 \right|^{u} \left| H_2 \right|^{u}}
\end{equation*}
for every $H_1, H_2 \in \GG_p$.
\end{proposition}

\begin{proof}
Choose $k \geq 1$ such that $p^k H_1 = p^k H_2 = 0$. Let $R = \Z / p^k \Z$, $A_n' \in \M_{n \times (n+u)}(R)$ be the reduction of $A_n$ modulo $p^k$, $V = R^n$, $W = R^{n+u}$ and $\left\{ e_1, \hdots, e_{n+u} \right\}$ be the standard basis of $W$. Then $v_i = A_n' e_i \in V$ ($1 \le i \le n+u$) are random independent $\alpha_n$-balanced vectors in $V$. Let
$$
B_n' = \begin{pmatrix}
b_1 & b_2 & \cdots & b_{n+u} \\
\end{pmatrix} \in \M_{n \times (n+u)}(R)
$$
be the reduction of $B_n$ modulo $p^k$. Then we have
\begin{equation} \label{eq31b2}
\begin{split}
& \EE(\# \Sur(\cok(A_n), H_1) \# \Sur(\cok(A_n + B_n), H_2)) \\
= \, & \sum_{\substack{F_1 \in \Sur_R(V, H_1) \\ F_2 \in \Sur_R(V, H_2)}} \prod_{i=1}^{n+u} \PP (F_1v_i = 0 \text{ and } F_2v_i = -F_2b_i).
\end{split}    
\end{equation}
Let $p_1 : H_1 \times H_2 \rightarrow H_1$ and $p_2 : H_1 \times H_2 \rightarrow H_2$ be projections and
$$
\LL = \left\{ G \leq H_1 \times H_2 \mid p_1(G)= H_1 \text{ and } p_2(G) = H_2 \right\}.
$$
For each $G \in \LL$, let $G_1$ (resp. $G_2$) be the subgroup of $H_1$ (resp. $H_2$) such that $G \cap (H_1 \times \left\{ 0 \right\}) = G_1 \times \left\{ 0 \right\}$ (resp. $G \cap (\left\{ 0 \right\} \times H_2) = \left\{ 0 \right\} \times G_2$). Also let
\begin{equation*}
\begin{split}
S_G & := \left\{ F \in \Sur_R(V, H_1) \times \Sur_R(V, H_2) \subset \Hom(V, H_1 \times H_2) \mid \im(F) = G \right\} \\
& \: \cong \Sur_R(V, G).
\end{split}
\end{equation*}
If $F = (F_1, F_2) \in S_G$, then $F_1(\ker F_2) = G_1$ and $F_2(\ker F_1) = G_2$. We also note that
\begin{equation*}
\begin{split}
\left| H_1 \right|\left| F_2(\ker F_1) \right| 
= \, & [V : \ker F_1][\ker F_1 : \ker F_1 \cap \ker F_2] \\
= \, & [V : \ker F_1 \cap \ker F_2] \\
= \, & [V : \ker F_2][\ker F_2 : \ker F_1 \cap \ker F_2] \\
= \, & \left| H_2 \right|\left| F_1(\ker F_2) \right|
\end{split}    
\end{equation*}
and $\ker F = \ker F_1 \cap \ker F_2$ so we have $\left| H_1 \right| \left| G_2 \right| = \left| G \right| = \left| H_2 \right| \left| G_1 \right|$, or equivalently, 
$$
[H_1 : G_1] = [H_2 : G_2] = [H_1 \times H_2 : G]
$$ 
for $G \in \LL$. 

For a given $F = (F_1, F_2) \in S_G$, assume that 
$$
\text{P}_{F, b} 
:= \prod_{i=1}^{n+u} \PP (Fv_i = (0, -F_2b_i)) 
= \prod_{i=1}^{n+u} \PP (F_1v_i = 0 \text{ and } F_2v_i = -F_2b_i) > 0.
$$
Then there are $w_1, \hdots, w_{n+u} \in V$ such that
$F_1w_i=0$ and $F_2w_i = -F_2b_i$ for each $i$. In this case, $F_2b_i = - F_2w_i \in F_2(\ker F_1)$ so $F_2(\left< b_i \right>) \leq F_2(\ker F_1)$. Similarly, $F_1(w_i+b_i) = F_1(b_i)$ and $F_2(w_i+b_i)=0$ so $F_1(\left< b_i \right>) \leq F_1(\ker F_2)$. Now define
$$
S_{G}^b := \left\{ F = (F_1, F_2) \in S_{G} \mid F_1(\left< b_i \right>) \leq G_1 \text{ and } F_2(\left< b_i \right>) \leq G_2  \right\}
$$
for $G \in \LL$. For a $F \in S_G$, $\text{P}_{F, b}>0$ only if $F \in S_G^b$ so we have
\begin{equation} \label{eq31b3}
\sum_{\substack{F_1 \in \Sur_R(V, H_1) \\ F_2 \in \Sur_R(V, H_2)}} \prod_{i=1}^{n+u} \PP (F_1v_i = 0 \text{ and } F_2v_i = -F_2b_i) 
= \sum_{G \in \LL} \sum_{F \in S_{G}}  \text{P}_{F, b} 
=\sum_{G \in \LL} \sum_{F \in S_{G}^{b}} \text{P}_{F, b}.
\end{equation}
Let $r = r_p(H_1)$, $r_n = r_p(B_n)$ and assume that $n$ is sufficiently large so that $r_n \geq r$. We compute the limit $\displaystyle \lim_{n \rightarrow \infty} \sum_{G \in \LL} \sum_{F \in S_{G}^{b}} \text{P}_{F, b} = \sum_{G \in \LL} \lim_{n \rightarrow \infty} \sum_{F \in S_{G}^{b}} \text{P}_{F, b}$ by computing $\displaystyle \lim_{n \rightarrow \infty} \sum_{F \in S_{G}^{b}} \text{P}_{F, b}$ for each $G \in \LL$. We proceed in two cases based on whether $G = H_1 \times H_2$ or not.
    
    \phantom{}

    \textbf{Case I.} $G = H_1 \times H_2$. $S_{H_1 \times H_2}^b = \Sur_R(V, H_1 \times H_2)$. Define $U_F \in \Hom(W, H_1 \times H_2)$ by $U_Fe_i = (0, -F_2b_i)$ for each $F \in \Sur_R(V, H_1 \times H_2)$. 
    Following the proof of \cite[Theorem 4.12]{NW22a} as in Proposition \ref{prop31b}, we can prove that there are constants $c_2, K_2>0$ (which are independent of $n$) such that
    $$
    \left | \sum_{F \in 
   S_{H_1 \times H_2}^b} \text{P}_{F, b} - \left | H_1 \times H_2 \right |^{-u} \right |
    = \left | \sum_{F \in 
    \Sur_R(V, H_1 \times H_2)} \PP(FA_n' = U_F) - \left | H_1 \times H_2 \right |^{-u} \right | \le K_2n^{-c_2}
    $$
    for every $n$ and an $\alpha_n$-balanced random matrix $A_n' \in \M_{n \times (n+u)}(R)$. Now we have
    \begin{equation} \label{eq31b4}
    \lim_{n \rightarrow \infty} \sum_{F \in S_{H_1 \times H_2}^{b}} \text{P}_{F, b}
    = \frac{1}{\left| H_1 \times H_2 \right|^{u}}
    = \frac{1}{\left| H_1 \right|^{u} \left| H_2 \right|^{u}}.
    \end{equation}
    
    \textbf{Case II.} $G \neq H_1 \times H_2$. Choose $b_{i_1}, \hdots, b_{i_{r_n}}$ such that $\left\{ b_{i_1}, \hdots, b_{i_{r_n}}, y_1, \hdots, y_{n-r_n} \right\}$ is a basis of $V$ for some $y_1, \hdots, y_{n-r_n} \in V$. For $(F_1, F_2) \in S_G^b$, $F_1$ is determined by $F_1 b_{i_j} \in G_1$ ($1 \leq j \leq r_n$) and $F_1 y_j \in H_1$ ($1 \leq j \leq n-r_n$) so 
    \begin{equation*} 
    \# \begin{Bmatrix}
F_1 \in \Sur(V, H_1) \mid (F_1, F_2) \in S_{G}^b \\
\text{for some } F_2 \in \Sur(V, H_2)
\end{Bmatrix}
    \leq \left| G_1 \right|^{r_n} \left| H_1 \right|^{n-r_n}.
    \end{equation*}
    Now we bound the number of $F_2 \in \Sur_R(V, H_2)$ such that $(F_1, F_2) \in S_G^b$ for a fixed $F_1 \in \Sur_R(V, H_1)$. 
    Since $V / \ker(F_1) \cong H_1 \cong \prod_{j=1}^{r} \Z/p^{\lambda_j} \Z$ ($k \geq \lambda_1 \geq \cdots \geq \lambda_r \geq 1$), there is a basis $\left\{ z_1, \hdots, z_n \right\}$ of $V$ such that $p^{\lambda_1}z_1, \hdots, p^{\lambda_r}z_r, z_{r+1}, \hdots, z_n$ generate $\ker(F_1)$. 
Then $F_2$ is determined by $F_2 z_j \in H_2$ ($1 \leq j \leq r$) and $F_2 z_j \in G_2$ ($r+1 \leq j \leq n$) so the number of possible choices of $F_2$ is bounded above by $\left| H_2 \right|^r \left| G_2 \right|^{n-r}$. Thus we have
    $$
    \left| S_G^b \right| 
    \leq \left| G_1 \right|^{r_n} \left| H_1 \right|^{n-r_n} \left| H_2 \right|^r \left| G_2 \right|^{n-r}
    = \left| G \right|^n \left ( \frac{\left| G \right|}{\left| H_1 \times H_2 \right|} \right )^{r_n-r}.
    $$
    If $F \in S_G^b$ is a code of distance $\delta n$ (as an element of $\Sur_R(V, G)$) for $\delta > 0$ in the sense of \cite[Definition 4.4]{NW22a}, then 
    $$
    \text{P}_{F, b} \leq \frac{K_1}{\left| G \right|^{n+u}}
    $$
    for some constant $K_1>0$ by \cite[Lemma 4.7]{NW22a}. Thus we have
    $$
    \sum_{\substack{F \in S_{G}^{b} \\ F \text{ code of distance } \delta n}} \text{P}_{F, b}
    \leq \frac{K_1 \left| S_G^b \right|}{\left| G \right|^{n+u}}
    \leq \frac{K_1 \left| G \right|^{-u}}{2^{r_n-r}} 
    $$
    and the assumption $\lim_{n \rightarrow \infty} r_n = \infty$ implies that
    \begin{equation} \label{eq31b5}
    \lim_{n \rightarrow \infty} \sum_{\substack{F \in S_{G}^{b} \\ F \text{ code of distance } \delta n}} \text{P}_{F, b} = 0.
    \end{equation}
    
    For each $F \in S_G^b$, we have $\text{P}_{F, b} = \PP(FA_n' = B_F^{b})$ where $B_F^{b} \in \Hom(W, G)$ is given by $B_F^{b}e_i = (0, -F_2b_i)$ for each $1 \le i \le n+u$. 
    Following the proof of the last equation in \cite[p. 23]{NW22a}, we can prove that there are constants $c_3, K_3>0$ (which are independent of $n$) such that
\begin{equation} \label{eqnew2}
\sum_{\substack{F \in S_{G}^{b} \\ F \text{ not code of distance } \delta n}}\text{P}_{F, b} 
= \sum_{\substack{F \in S_{G}^{b} \subset \Sur_R(V, G) \\ F \text{ not code of distance } \delta n}} \PP(FA_n' = B_F^{b}) 
\leq K_3 n^{-c_3}.
\end{equation}
    for all sufficiently small $\delta = \delta_{G} > 0$. (To do this, we need to generalize \cite[Lemma 4.11]{NW22a} to an upper bound of $\PP(FX=A)$ for every $A \in \im (F)$. The proof for $A=0$ works for every $A \in \im (F)$.) The equation (\ref{eqnew2}) implies that
    \begin{equation} \label{eq31b6}
    \lim_{n \rightarrow \infty} \sum_{\substack{F \in S_{G}^{b} \\ F \text{ not code of distance } \delta n}} \text{P}_{F, b} = 0
    \end{equation}
    for all sufficiently small $\delta = \delta_{G} > 0$.

By the equations (\ref{eq31b4}), (\ref{eq31b5}) and (\ref{eq31b6}), we have
\begin{equation} \label{eq31b7}
\lim_{n \rightarrow \infty} \sum_{G \in \LL} \sum_{F \in S_{G}^{b}} \text{P}_{F, b} = \frac{1}{\left| H_1 \right|^{u} \left| H_2 \right|^{u}}.
\end{equation}
We conclude the proof from the equations (\ref{eq31b2}), (\ref{eq31b3}) and (\ref{eq31b7}).
\end{proof}

\begin{theorem} \label{thm31e}
Let $u$, $A_n$ and $(B_n)_{n \geq 1}$ be as in Proposition \ref{prop31d}. Then we have
\begin{equation*} 
\lim_{n \rightarrow \infty} \PP \begin{pmatrix}
\cok(A_n) \cong H_1 \text{ and} \\
\cok(A_n+B_n) \cong H_2
\end{pmatrix} = \prod_{i=1}^{2} \frac{\prod_{k=1}^{\infty} (1-p^{-k-u})}{\left| H_i \right|^{u} \left| \Aut(H_i) \right| }
\end{equation*}
for every $H_1, H_2 \in \GG_p$.
\end{theorem}

\begin{proof}
Let $Y_p(u)$ be a random finite abelian $p$-group such that 
$$
\PP(Y_p(u) \cong B) = \frac{\prod_{k=1}^{\infty} (1-p^{-k-u})}{\left| B \right|^{u} \left| \Aut(B) \right|}
$$
for every $B \in \GG_p$. By Proposition \ref{prop31d} and \cite[Lemma 3.2]{Woo19}, we have
\begin{equation*}
\begin{split}
& \lim_{n \rightarrow \infty} \EE(\# \Sur(\cok(A_n), H_1) \# \Sur(\cok(A_n + B_n), H_2)) \\
= \, & \frac{1}{\left| H_1 \right|^{u} \left| H_2 \right|^{u}} \\
= \, & \EE(\# \Sur(Y_p(u), H_1)) \EE(\# \Sur(Y_p(u), H_2)).
\end{split}    
\end{equation*}
Now Theorem \ref{thm1c} finishes the proof.
\end{proof}

\subsection{Joint distribution of $\cok(A)$ and $\cok(A+pI_n)$} \label{Sub32}

In this section, we compute the limiting joint distribution of $\cok(A_n)$ and $\cok(A_n+pI_n)$ for random matrices $A_n \in \M_n(\Z_p)$. The following notation will be used throughout this section.
\begin{itemize}
    \item $\mu_n$ is the Haar probability measure on $\M_n(\Z_p)$.
    
    \item Let $E$ be the set of partitions. For $\lambda = (\lambda_1 \geq \cdots \geq \lambda_r) \in E$, let $\ell(\lambda) = r$, $\left| \lambda \right| = \sum_{i=1}^{r} \lambda_i$ and $G_{\lambda} = \prod_{i=1}^{r} \Z / p^{\lambda_i} \Z$ be the finite abelian $p$-group of type $\lambda$. 
    
    \item Let $E^{(r)} = \left\{ \lambda \in E \mid \ell(\lambda) \leq r \right\}$, $E_X = \left\{ \lambda \in E \mid \left| \lambda \right| \leq X \right\}$ and $\GG_p(X) = \left\{ G_{\lambda} \mid \lambda \in E_X \right\} = \left\{ H \in \GG_p \mid \left| H \right| \leq p^X \right\}$.

    \item For $\lambda = (\lambda_1 \geq \cdots \geq \lambda_r) \in E$, define $D_{\lambda} = \diag(p^{\lambda_1}, \hdots, p^{\lambda_r}) \in \M_r(\Z_p)$. For $n \geq r$, let $D_{\lambda, n} = \begin{pmatrix}
D_{\lambda} & O \\
O & I_{n-r} \\
\end{pmatrix} \in \M_n(\Z_p)$ and 
$$
[D_{\lambda, n}] = \GL_n(\Z_p)D_{\lambda, n}\GL_n(\Z_p) \subset \M_n(\Z_p).
$$

    \item $\M_n(\Z_p)^{\neq 0} := \left \{ A \in \M_n(\Z_p) \mid \det (A) \neq 0 \right \} = \underset{\lambda \in E^{(n)}}{\bigsqcup} [D_{\lambda, n}]$ (Smith normal form). For $n \geq X$, define $\M_n(\Z_p)^{\neq 0}_{X} := \underset{\lambda \in E_{X}}{\bigsqcup} [D_{\lambda, n}] \subset \M_n(\Z_p)^{\neq 0}$.

    \item For $n \geq 2r$, let
    $$
\widetilde{\M}_{n, r}(\Z_p) := \left \{ \begin{pmatrix}
A_1 & A_2 & A_3\\ 
I_r & A_4 & A_5\\ 
O & A_6 & A_7
\end{pmatrix} \in \M_{r+r+(n-2r)}(\Z_p) \right \} \subset \M_n(\Z_p).
$$
Note that $A_1, A_2, A_4 \in \M_r(\Z_p)$, $A_3, A_5 \in \M_{r \times (n-2r)}(\Z_p)$, $A_6 \in \M_{(n-2r) \times r}(\Z_p)$ and $A_7 \in \M_{n-2r}(\Z_p)$.

    \item For a uniform random matrix $A \in \M_n(\Z_p)$ and a given $B \in \M_n(\Z_p)$, define
    $$
    P_{B}(H_1, H_2) := \PP \begin{pmatrix}
\cok(A) \cong H_1 \text{ and} \\ 
\cok(A+pB) \cong H_2
\end{pmatrix}.
    $$
    For $n \geq 2r$, a uniform random matrix $\widetilde{A} \in \widetilde{\M}_{n, r}(\Z_p)$ and a given $B \in \M_n(\Z_p)$, define
    $$
\widetilde{P}_{B, r}(H_1, H_2) := \PP \begin{pmatrix}
\cok(\widetilde{A}) \cong H_1 \text{ and} \\ 
\cok(\widetilde{A}+pB) \cong H_2
\end{pmatrix}.
    $$
\end{itemize}

\begin{lemma} \label{lem32a}
$$
\lim_{X \rightarrow \infty} \lim_{n \rightarrow \infty} \mu_n(\M_n(\Z_p)^{\neq 0}_{X}) = 1.
$$
\end{lemma}

\begin{proof}
Since $E_X$ is a finite set, we have
\begin{align*}
\lim_{X \rightarrow \infty} \lim_{n \rightarrow \infty} \mu_n(\M_n(\Z_p)^{\neq 0}_{X})
& = \lim_{X \rightarrow \infty} \lim_{n \rightarrow \infty} \sum_{\lambda \in E_X} \mu_n([D_{\lambda, n}]) \\
& = \lim_{X \rightarrow \infty} \sum_{\lambda \in E_X} \lim_{n \rightarrow \infty} \mu_n([D_{\lambda, n}]) \\
& = \lim_{X \rightarrow \infty} \sum_{H \in \GG_p(X)}  \lim_{n \rightarrow \infty} \underset{A \in \M_n(\Z_p)}{\PP} (\cok (A) \cong H) \\
& = \sum_{H \in \GG_p} \frac{c_{\infty}(p)}{\left| \Aut(H) \right|} \\
& = 1. 
\end{align*}
\end{proof}

The proof of the next lemma is similar to \cite[Proposition 4.2]{Lee23a}.

\begin{lemma} \label{lem32b}
For every $\lambda \in E$ and $H_1, H_2 \in \GG_p$, we have
$$
\lim_{n \rightarrow \infty} \left| P_{D_{\lambda, n}}(H_1, H_2) - P_{I_n}(H_1, H_2) \right| = 0.
$$
\end{lemma}

\begin{proof}
Fix $H_1, H_2 \in \GG_p$, $r \geq 0$ and assume that $n>2r$. For every $\lambda \in E^{(r)}$, we obtain
\begin{equation} \label{eq32a}
\left| P_{D_{\lambda, n}}(H_1, H_2) - \widetilde{P}_{D_{\lambda, n}, r}(H_1, H_2) \right| \leq 1 - c_{n-r, r} := 1 - \prod_{j=0}^{r-1}(1 - \frac{1}{p^{n-r-j}})
\end{equation}
by applying \cite[Lemma 2.3]{Lee23a} as in \cite[Section 2.2]{Lee23a}. Let
$$
A = \begin{pmatrix}
A_1 & A_2 & A_3 \\ 
I_r & A_4 & A_5 \\ 
O & A_6 & A_7 \\
\end{pmatrix} \in \widetilde{\M}_{n, r}(\Z_p).
$$
By applying the transformation $(A, B) \, \Rightarrow (A', B')$ given in the proof of \cite[Proposition 4.2]{Lee23a}, we can simultaneously transform $(A, D_{\lambda, n})$ to
$$
\small
(A', D') = \left ( \begin{pmatrix}
O & A_2' & A_3' \\
I_r & O & O \\
O & A_6 & A_7 \\
\end{pmatrix}, \, \begin{pmatrix}
D_{\lambda, r} & -A_1' & O \\
O & I_r & O \\
O & O & I_{n-2r} \\
\end{pmatrix}  \right )
$$
for 
$$
(A_1', A_2', A_3') := (A_1+D_{\lambda, r}A_4, \, A_2-A_1A_4-D_{\lambda, r}A_5A_6, \, A_3-A_1A_5 +D_{\lambda, r}A_5A_7).
$$
Now we prove that the matrices $A_1'$, $A_2'$, $A_3'$, $A_6$ and $A_7$ are uniform and independent. 
\begin{itemize}
    \item For a given $D_{\lambda, r}$, $A_4$ and $A_5$, the matrices $A_1$, $A_2$, $A_3$, $A_6$ and $A_7$ are uniform and independent.

    \item $A_3' = A_3-A_1A_5 +D_{\lambda, r}A_5A_7$ so the matrices $A_1$, $A_2$, $A_3'$, $A_6$ and $A_7$ are uniform and independent.

    \item $A_2'=A_2-A_1A_4-D_{\lambda, r}A_5A_6$ so the matrices $A_1$, $A_2'$, $A_3'$, $A_6$ and $A_7$ are uniform and independent.

    \item $A_1' = A_1+D_{\lambda, r}A_4$ so the matrices $A_1'$, $A_2'$, $A_3'$, $A_6$ and $A_7$ are uniform and independent.
\end{itemize}
Now we have
$$
\cok(A) \cong \cok (A') \cong \cok \begin{pmatrix}
A_2' & A_3' \\
A_6 & A_7 \\
\end{pmatrix}
$$
and
\begin{align*}
\small
\cok(A+pD_{\lambda, n}) & \cong \cok(A'+pD') \\
& = \cok \begin{pmatrix}
pD_{\lambda, r} & A_2'-pA_1' & A_3' \\ 
I_r & pI_r & O \\ 
O & A_6 & A_7+pI_{n-2r}
\end{pmatrix} \\
& \cong \cok \begin{pmatrix}
O & A_2'-pA_1'-p^2D_{\lambda, r} & A_3' \\ 
I_r & pI_r & O \\ 
O & A_6 & A_7+pI_{n-2r}
\end{pmatrix} \\
& \cong \cok \left ( \begin{pmatrix}
A_2' & A_3' \\
A_6 & A_7 \\
\end{pmatrix} + p \begin{pmatrix}
-A_1'' & O \\
O & I_{n-2r} \\
\end{pmatrix}  \right ).
\end{align*}
Since $A_1'' = A_1'+pD_{\lambda, r}$, the matrices $A_1''$, $A_2'$, $A_3'$, $A_6$ and $A_7$ are uniform and independent. Thus the probability $\widetilde{P}_{D_{\lambda, n}, r}(H_1, H_2)$ is independent of the choice of $\lambda \in E^{(r)}$ and the equation (\ref{eq32a}) implies that
\begin{equation*}
\left| P_{D_{\lambda, n}}(H_1, H_2) - P_{I_n}(H_1, H_2) \right| \leq 2(1 - c_{n-r, r})
\end{equation*}
for every $\lambda \in E^{(r)}$, which finishes the proof.
\end{proof}

\begin{lemma} \label{lem32c}
Let $A_n, B_n \in \M_n(\Z_p)$ be uniform random matrices for each $n$. Then we have
\begin{equation*} 
\lim_{n \rightarrow \infty} \left| \PP \begin{pmatrix}
\cok(A_n) \cong H_1 \text{ and} \\ 
\cok(A_n+pB_n) \cong H_2
\end{pmatrix} - P_{I_n}(H_1, H_2) \right| = 0
\end{equation*}
for every $H_1, H_2 \in \GG_p$.
\end{lemma}

\begin{proof}
Fix $X>0$ and denote $(A_n, B_n)=(A, B)$ for simplicity. For any $n \geq X$, we have
\begin{equation*}
\small
\begin{split}
& \PP \begin{pmatrix}
\cok(A) \cong H_1 \text{ and} \\ 
\cok(A+pB) \cong H_2
\end{pmatrix} \\
= & \sum_{\lambda \in E_X} (\mu_n \times \mu_n) \left ( \left \{ (A, B) \in \M_n(\Z_p) \times [D_{\lambda, n}] \mid \begin{matrix}
\cok(A) \cong H_1 \text{ and} \\ 
\cok(A+pB) \cong H_2
\end{matrix} \right \} \right ) \\
+ & (\mu_n \times \mu_n) \left ( \left \{ (A, B) \in \M_n(\Z_p) \times (\M_n(\Z_p) \setminus \M_n(\Z_p)^{\neq 0}_{X} ) \mid \begin{matrix}
\cok(A) \cong H_1 \text{ and} \\ 
\cok(A+pB) \cong H_2
\end{matrix} \right \} \right ) \\
= & \sum_{\lambda \in E_X} \mu_n([D_{\lambda, n}])P_{D_{\lambda, n}}(H_1, H_2) + \varepsilon_{n, X}(H_1, H_2)
\end{split}    
\end{equation*}
with 
\begin{equation} \label{eqnew1}
0 \leq \varepsilon_{n, X}(H_1, H_2) \leq 
1 - \mu_n(\M_n(\Z_p)^{\neq 0}_{X}).
\end{equation}
Since $E_X$ is a finite set, we have
\begin{equation*}
\begin{split}
& \lim_{n \rightarrow \infty} \left| \PP \begin{pmatrix}
\cok(A) \cong H_1 \text{ and} \\ 
\cok(A+pB) \cong H_2
\end{pmatrix} - P_{I_n}(H_1, H_2) \right| \\
\leq & \sum_{\lambda \in E_X} \lim_{n \rightarrow \infty} \mu_n([D_{\lambda, n}]) \left| P_{D_{\lambda, n}}(H_1, H_2) - P_{I_n}(H_1, H_2) \right| \\
+ & \lim_{n \rightarrow \infty} \left| \varepsilon_{n, X}(H_1, H_2) - (1 - \mu_n(\M_n(\Z_p)^{\neq 0}_{X}))P_{I_n}(H_1, H_2) \right| \\
\leq & \lim_{n \rightarrow \infty} (1-\mu_n(\M_n(\Z_p)^{\neq 0}_{X}))
\end{split}    
\end{equation*}
by Lemma \ref{lem32b} and the equation (\ref{eqnew1}). Taking the limit $X \rightarrow \infty$, Lemma \ref{lem32a} finishes the proof.
\end{proof}

The following proposition is a special case of Theorem \ref{thm32f}. Recall that $c_r(p) := \prod_{k=1}^{r} (1-p^{-k})$ and $c_{\infty}(p) := \prod_{k=1}^{\infty} (1-p^{-k})$. 

\begin{proposition} \label{prop32d}
Let $A_n \in \M_n(\Z_p)$ be a uniform random matrix for each $n$. Then we have
\begin{equation*}
\lim_{n \rightarrow \infty} \PP \begin{pmatrix}
\cok(A_n) \cong H_1 \text{ and} \\
\cok(A_n+pI_n) \cong H_2
\end{pmatrix} = \left\{\begin{matrix}
0 & (r_p(H_1) \neq r_p(H_2)) \\
\frac{p^{r^2} c_{\infty}(p)c_{r}(p)^2}{\left| \Aut(H_1) \right| \left| \Aut(H_2) \right|} & (r_p(H_1) = r_p(H_2) = r)
\end{matrix}\right.
\end{equation*}
for every $H_1, H_2 \in \GG_p$.
\end{proposition}

\begin{proof}
Since $\cok(A_n)$ and $\cok(A_n+pI_n)$ have the same $p$-rank, the case $r_p(H_1) \neq r_p(H_2)$ is clear. Now assume that $r_p(H_1) = r_p(H_2) = r$. By Lemma \ref{lem32c}, it is enough to show that
$$
\lim_{n \rightarrow \infty} \PP \begin{pmatrix}
\cok(A_n) \cong H_1 \text{ and} \\
\cok(A_n+pB_n) \cong H_2
\end{pmatrix} = \frac{p^{r^2} c_{\infty}(p)c_{r}(p)^2}{\left| \Aut(H_1) \right| \left| \Aut(H_2) \right|}
$$
for uniform random matrices $A_n, B_n \in \M_n(\Z_p)$. Denote $(A_n, B_n)=(A,B)$ for simplicity. 

Let $H_2$ be of type $\lambda$. By \cite[Lemma 7.2]{Woo17}, we have
$$
\left| \Aut(H_2) \right| = p^{\sum_{i=1}^{\lambda_1} (\lambda_i')^2} \prod_{i=1}^{\lambda_1} c_{\lambda_i'-\lambda_{i+1}'}(p).
$$
If $pH_2$ is of type $\mu$, then its conjugate is given by $\mu' = (\lambda_2' \geq \cdots \geq \lambda_{\lambda_1}')$ so we have
$$
\left| \Aut(pH_2) \right| = p^{\sum_{i=2}^{\lambda_1} (\lambda_i')^2}  \prod_{i=2}^{\lambda_1} c_{\lambda_i'-\lambda_{i+1}'}(p)
= \frac{\left| \Aut(H_2) \right|}{p^{r^2} c_{r-t}(p)}
$$
for $t = \lambda_2' = r_p(pH_2)$ again by \cite[Lemma 7.2]{Woo17}. This implies that
\begin{equation*}
\begin{split}
\frac{p^{r^2} c_{\infty}(p)c_{r}(p)^2}{\left| \Aut(H_1) \right| \left| \Aut(H_2) \right|} 
& = \frac{c_{\infty}(p)}{\left| \Aut(H_1) \right|} \cdot 
\frac{1}{\left| \Aut(pH_2) \right|}\frac{c_{r}(p)^2}{c_{r-t}(p)} \\
& = \lim_{n \rightarrow \infty} \PP (\cok(A) \cong H_1) \times \PP (\cok(C) \cong pH_2)
\end{split}
\end{equation*}
for a uniform random matrix $C \in \M_r(\Z_p)$ by \cite[Proposition 1]{FW89}. If $H_1$ is of type $\nu$, then we have $\ell (\nu) = r$. It remains to show that the conditional probability
\begin{equation*}
\begin{split}
& \PP(\cok(A+pB) \cong H_2 \mid \cok(A) \cong H_1) \\
= \, & \PP(\cok(A+pB) \cong H_2 \mid A \in [D_{\nu, n}]) \\
= \, & \PP(\cok(D_{\nu, n}+pB) \cong H_2)
\end{split}
\end{equation*}
is equal to the probability $\PP (\cok(C) \cong pH_2)$ for every $n > r$. 

For a uniform random matrix $B = \begin{pmatrix}
B_1 & B_2 \\
B_3 & B_4 \\
\end{pmatrix} \in \M_{r+(n-r)}(\Z_p)$, we have
\begin{equation*}
\small
\begin{split}
\cok(D_{\nu, n}+pB)
& = \cok \left ( \begin{pmatrix}
p(B_1+D) & pB_2 \\
pB_3 & pB_4+I_{n-r} \\
\end{pmatrix} \right ) \;\; (D = p^{-1}D_{\nu} \in \M_r(\Z_p))) \\
& \cong \cok \left ( \begin{pmatrix}
pB_1' & O \\
pB_3 & pB_4+I_{n-r} \\
\end{pmatrix} \right ) \;\; (B_1' = B_1+D-pB_2(pB_4+I_{n-r})^{-1}B_3) \\
& \cong \cok(pB_1')
\end{split}    
\end{equation*}
and $B_1' \in \M_r(\Z_p)$ is also a uniform random matrix since $B_1$ is uniform and the distributions of $B_2$, $B_3$, $B_4$ and $D$ are independent to the distribution of $B_1$.
Thus we have
\begin{equation*}
\PP(\cok(D_{\nu, n}+pB) \cong H_2)
= \PP(\cok(pB_1') \cong H_2) 
= \PP(\cok(B_1') \cong pH_2),
\end{equation*}
which finishes the proof.
\end{proof}

For a prime $p$ and integers $r_1, r_2 \geq 0$, let $p_1 : \F_p^{r_1} \times \F_p^{r_2} \rightarrow \F_p^{r_1}$ and $p_2 : \F_p^{r_1} \times \F_p^{r_2} \rightarrow \F_p^{r_2}$ be projections and $N(r_1, r_2)$ be the number of the $\F_p$-subspaces $W$ of $\F_p^{r_1} \times \F_p^{r_2}$ such that $p_1(W) = \F_p^{r_1}$ and $p_2(W) = \F_p^{r_2}$. 

\begin{proposition} \label{prop32e}
Let $A_n \in \M_n(\Z_p)$ be as in Theorem \ref{thm1a}. Then we have
\begin{equation*}
\lim_{n \rightarrow \infty} \EE(\# \Sur(\cok(A_n), H_1) \# \Sur(\cok(A_n + pI_n), H_2)) 
= N(r_p(H_1), r_p(H_2))
\end{equation*}
for every $H_1, H_2 \in \GG_p$.
\end{proposition}

\begin{proof}
We use the notation in the proof of Proposition \ref{prop31d}, with an additional condition that $u=0$. Define $r_j = r_p(H_j)$ and
\begin{equation*}
\LL_1 = \left\{ G \in \LL \mid pH_1 \leq G_1 \text{ and } pH_2 \leq G_2 \right\}.
\end{equation*}
Let $F=(F_1, F_2) \in S_G$ for $G \in \LL$ and assume that there are $v_1, \hdots, v_n \in V$ such that $F_1v_i = 0$ and $F_2v_i = -pF_2e_i$ for each $i$. Then $pF_2e_i \in F_2(\ker F_1) = G_2$ for each $i$ so $pH_2 \leq G_2$. By the same reason, we have $pH_1 \leq G_1$ so $G \in \LL_1$. Thus we have
\begin{equation*}
\begin{split}
& \EE(\# \Sur(\cok(A_n), H_1) \# \Sur(\cok(A_n + pI_n), H_2)) \\
= \, & \sum_{\substack{F_1 \in \Sur_R(V, H_1) \\ F_2 \in \Sur_R(V, H_2)}} \prod_{i=1}^{n} \PP (F_1 v_i = 0 \text{ and } F_2 v_i = -pF_2 e_i) \\
= \, & \sum_{G \in \LL_1} \sum_{F \in S_G} \prod_{i=1}^{n} \PP (Fv_i = (0, -pF_2 e_i)).
\end{split}    
\end{equation*}
Define $U_F \in \Hom(W, H_1 \times H_2)$ by $U_Fe_i = (0, -pF_2e_i)$ for each $F \in S_G$. 
Following the proof of \cite[Theorem 4.12]{NW22a} as in Proposition \ref{prop31b} and \ref{prop31d}, we can prove that for each $G \in \LL_1$ there are constants $c_{G, 2}, K_{G, 2}>0$ (which are independent of $n$) such that
$$
\left | \sum_{F \in \Sur_R(V, G)} \prod_{i=1}^{n} \PP (Fv_i = (0, -pF_2 e_i)) - 1 \right |
= \left | \sum_{F \in \Sur_R(V, G)} \PP(FA_n' = U_F) - 1 \right | \le K_{G,2}n^{-c_{G,2}}
$$
for every $n$ and an $\alpha_n$-balanced random matrix $A_n' \in \M_{n}(R)$. Now we have
\begin{equation*} 
\lim_{n \rightarrow \infty} \sum_{F \in S_G} \prod_{i=1}^{n} \PP (Fv_i = (0, -pF_2 e_i)) 
= \lim_{n \rightarrow \infty} \sum_{F \in \Sur_R(V, G)} \prod_{i=1}^{n} \PP (Fv_i = (0, -pF_2 e_i))
= 1
\end{equation*}
for each $G \in \LL_1$. Since the map
$$
\LL_1 \rightarrow \left\{ W \leq \F_p^{r_1} \times \F_p^{r_2} \mid p_1(W) = \F_p^{r_1}, \, p_2(W) = \F_p^{r_2} \right\} \;\; (G \mapsto G/(pH_1 \times pH_2))
$$
is a bijection, we have $\left| \LL_1 \right| = N(r_1, r_2)$. This finishes the proof.
\end{proof}

\begin{theorem} \label{thm32f}
Let $A_n$ be as in Theorem \ref{thm1a}. Then we have
\begin{equation} \label{eq32e}
\lim_{n \rightarrow \infty} \PP \begin{pmatrix}
\cok(A_n) \cong H_1 \text{ and} \\
\cok(A_n+pI_n) \cong H_2
\end{pmatrix} = \left\{\begin{matrix}
0 & (r_p(H_1) \neq r_p(H_2)) \\
\frac{p^{r^2} c_{\infty}(p)c_{r}(p)^2}{\left| \Aut(H_1) \right| \left| \Aut(H_2) \right|} & (r_p(H_1) = r_p(H_2) = r)
\end{matrix}\right.
\end{equation}
for every $H_1, H_2 \in \GG_p$.
\end{theorem}

\begin{proof}
The number $N(r_1, r_2)$ is bounded above by the number of subspaces of $\F_p^{r_1+r_2}$, so
$$
N(r_1, r_2)
\leq \sum_{k=0}^{r_1+r_2} \prod_{i=0}^{k-1} \frac{p^{r_1+r_2}-p^i}{p^k - p^i}
= O\left ( \sum_{k=0}^{r_1+r_2} p^{k(r_1+r_2-k)} \right )
= O(p^{\frac{r_1^2+r_2^2}{2}}).
$$
Recall that $m(H_j) = p^{\sum_{i} \frac{\lambda(j)_i'^2}{2}}$ if $H_j$ is of type $\lambda(j)$. Since $\lambda(j)_1' = r_p(H_j)$, we have 
$$
N(r_p(H_1), r_p(H_2)) = O(m(H_1)m(H_2))
$$
so Theorem \ref{thm1c} can be applied to the mixed moments $N(r_p(H_1), r_p(H_2))$. Now the equation (\ref{eq32e}) follows from Theorem \ref{thm1c}, Proposition \ref{prop32d} and \ref{prop32e}.
\end{proof}

\section{Joint distribution of random groups} \label{Sec4}

In this section, we compute the joint distribution of random non-abelian groups. First we compute the moments of the random group $X_u \in \GG$ whose probability distribution is defined by the measure $\mu_u$ given in Theorem \ref{thm1h}. The moments of $X_0$ were computed in \cite[Lemma 3.20]{Woo22}, following the strategy of \cite[Theorem 6.2]{LWZ19}. The proof for an arbitrary $u$ is almost identical, so we only give a sketch here.

\begin{lemma} \label{lem4a}
Let $u \geq 0$ be an integer and $X_u$ be the random group defined as above. Then for any finite group $H$, we have
$$
\EE (\# \Sur(X_u, H)) = \left| H \right|^{-u}.
$$
\end{lemma}

\begin{proof}
Let $Z_n$ be the random profinite group $F_n/\left< r_1, \hdots, r_{n+u} \right>$, where $r_i$ are independent Haar random elements of $F_n$. For any positive integer $\ell$, let $\CC_{\ell}$ be the set of finite groups of order at most $\ell$ and $G$ be a finite group of level-$\CC_{\ell}$. We define
$$
f_n(G, \ell) = \EE(\# \Sur(Z_n, H) \times \mathbf{1}_{Z_n^{\CC_{\ell}} \cong G}),
$$
where $\mathbf{1}_{Z_n^{\CC_{\ell}} \cong G}$ denotes the indicator function of $Z_n^{\CC_{\ell}} \cong G$. Now we check the conditions given in \cite[Lemma 5.10]{LWZ19} are satisfied. Following the proof of \cite[Lemma 3.20]{Woo22}, we have
$$
f_n(G, \ell) = g_n(G, \ell) P_{u, n}(U_{\CC_{\ell}, G})
$$
where $P_{u, n}(U_{\CC_{\ell}, G})$ is defined as in \cite[p. 146]{LW20} and
\begin{equation*}
\small
\begin{split}
g_n(G, \ell) & = \frac{\left| H^{\CC_{\ell}} \right|^{n+u}}{\left| H \right|^{n+u}} \sum_{\phi \in \Sur(F_n, H)}\frac{\# \left\{ (\tau, \pi) \in \Sur(F_n^{\CC_{\ell}}, G) \times \Sur(G, H^{\CC_{\ell}}) \mid \pi \circ \tau = \overline{\phi} \right\}}{\left| \Aut(G) \right| \left| G \right|^{n+u}} \\
& = \frac{\left| H^{\CC_{\ell}} \right|^{n+u} }{\left| \Aut(G) \right| \left| G \right|^{n+u} \left| H \right|^{n+u}} 
\sum_{\substack{\tau \in \Sur(F_n^{\CC_{\ell}}, G) \\ \pi \in \Sur(G, H^{\CC_{\ell}})}} \# \Sur(\pi \circ \tau, \pi_H).
\end{split}    
\end{equation*}
($\Sur(\pi \circ \tau, \pi_H)$ denotes the set of $\phi \in \Sur(F_n, H)$ which induces $\pi \circ \tau \in \Sur(F_n^{\CC_{\ell}}, H^{\CC_{\ell}})$.) We have
$$
g(G, \ell) := \lim_{n \rightarrow \infty} g_n(G, \ell) = \frac{\left| H^{\CC_{\ell}} \right|^{u} \# \Sur(G, H^{\CC_{\ell}})}{\left| \Aut(G) \right| \left| G \right|^{u} \left| H \right|^{u}}
$$
and $g_n(G, \ell) \leq g(G, \ell)$ for every $n$ so the condition (2) is satisfied. The condition (3) follows from the definition of $f_n(G, \ell)$. Now \cite[Lemma 5.10]{LWZ19} implies that
\begin{equation} \label{eq41a}
\sum_{\substack{G \\ G^{\CC_{\ell}} \cong G}} \lim_{n \rightarrow \infty} f_n(G, \ell) 
= \lim_{n \rightarrow \infty} f_n(\text{trivial group}, 1)
= \lim_{n \rightarrow \infty} \EE(\# \Sur(Z_n, H))
= \left| H \right|^{-u}.
\end{equation}
For a sufficiently large $\ell$ such that $H^{\CC_{\ell}} \cong H$, we have
\begin{equation} \label{eq41b}
\begin{split}
\lim_{n \rightarrow \infty} f_n(G, \ell)
& = \lim_{n \rightarrow \infty} \# \Sur(G, H) \PP(Z_n^{\CC_{\ell}} \cong G) \\
& = \# \Sur(G, H) \PP(X_u^{\CC_{\ell}} \cong G)
\end{split}    
\end{equation}
by Theorem \ref{thm1h}. Now the equations (\ref{eq41a}) and (\ref{eq41b}) imply that $\EE(\# \Sur(X_u, H)) = \left| H \right|^{-u}$.
\end{proof}

\begin{lemma} \label{lem4b}
Let $H_1$ and $H_2$ be finite groups. For $G_1 \trianglelefteq H_1$ and $G_2 \trianglelefteq H_2$, define
    $$
    S_{G_1, G_2} := \begin{Bmatrix}
(\phi_1, \phi_2) \in \Sur(F_n, H_1) \times \Sur(F_n, H_2) \mid \\
\phi_1(\ker \phi_2) = G_1 \text{ and } \phi_2(\ker \phi_1) = G_2
\end{Bmatrix}.
    $$
    Then $\left| S_{G_1, G_2} \right| \leq \left| H_1 \right|^n \left| G_2 \right|^n \left| H_2 \right|^r$ for $r = \rank(H_1)$.
\end{lemma}

\begin{proof}
For a fixed $\phi_1 \in \Sur(F_n, H_1)$, we bound the number of $\phi_2 \in \Sur(F_n, H_2)$ such that $(\phi_1, \phi_2) \in S_{G_1, G_2}$. Choose a generating set $\left\{ u_1, \hdots, u_r \right\}$ of a group $H_1$ and $y_j \in F_n$ such that $\phi_1(y_j)=u_j$ for $1 \leq j \leq r$. Let $\left\{ x_1, \hdots, x_n \right\}$ be a generating set of $F_n$. Since 
$$
\phi_1(x_i) \in H_1 = \left< \phi_1(y_1), \hdots, \phi_1(y_r) \right>,
$$
we can write $x_i = k_i z_i$ for some $k_i \in \ker \phi_1$ and $z_i \in \left< y_1, \hdots, y_r \right>$. Now we have
$$
F_n 
= \left< x_1, \hdots, x_n \right> 
\leq \left< k_1, \hdots, k_n, z_1, \hdots, z_n \right> 
\leq \left< k_1, \hdots, k_n, y_1, \hdots, y_r \right> 
$$
so $\phi_2$ is determined by $\phi_2(k_1), \hdots, \phi_2(k_n) \in G_2$ and $\phi_2(y_1), \hdots, \phi_2(y_r) \in H_2$. This implies that the number of possible choices of $\phi_2$ such that $(\phi_1, \phi_2) \in S_{G_1, G_2}$ (for a fixed $\phi_1$) is bounded above by $\left| G_2 \right|^n \left| H_2 \right|^r$. Since the number of possible choices of $\phi_1$ is bounded above by $\left| H_1 \right|^n$, we obtain that $\left| S_{G_1, G_2} \right| \leq \left| H_1 \right|^n \left| G_2 \right|^n \left| H_2 \right|^r$. 
\end{proof}

For $\phi_1 \in \Sur(F_n, H_1)$ and $\phi_2 \in \Sur(F_n, H_2)$, it is easy to show that $\phi_1(\ker \phi_2)$ is a normal subgroup of $H_1$ and $\phi_2(\ker \phi_1)$ is a normal subgroup of $H_2$. We also have
\begin{equation*}
\begin{split}
H_1 / \phi_1(\ker \phi_2) 
& \cong (F_n / \ker \phi_1)/(\left< \ker \phi_1, \ker \phi_2 \right> / \ker \phi_1) \\
& \cong (F_n / \ker \phi_2)/(\left< \ker \phi_1, \ker \phi_2 \right> / \ker \phi_2) \\
& \cong H_2 / \phi_2(\ker \phi_1) 
\end{split}    
\end{equation*}
so the set $S_{G_1, G_2}$ is empty if $H_1/G_1$ and $H_2/G_2$ are not isomorphic. The following proposition is a non-abelian analogue of Proposition \ref{prop31d}.

\begin{proposition} \label{prop4c}
Let $u \geq 0$ be an integer, $r_1, \hdots, r_{n+u}$ be independent uniform random elements of $F_n$ and $b_{n, 1}, \hdots, b_{n, n+u}$ be given elements of $F_n$ for each $n$. Assume that $\lim_{n \rightarrow \infty} d_n = \infty$, where $d_n$ is the maximum size of a subset $S \subset \left< b_{n, 1}, \hdots, b_{n, n+u} \right>$ which can be extended to a generating set of $F_n$. Then for any finite groups $H_1$ and $H_2$, we have
\small
\begin{align*}
& \lim_{n \rightarrow \infty} \EE( \# \Sur(F_n/\left< r_1, \hdots, r_{n+u} \right>, H_1) \# \Sur(F_n/\left< r_1b_{n,1}, \hdots, r_{n+u}b_{n, n+u} \right>, H_2) ) \\
= \, & \frac{1}{\left| H_1 \right|^{u} \left| H_2 \right|^{u}}. 
\end{align*}
\end{proposition}

\begin{proof}
To ease the notation, we write $b_{n, i}$ as $b_i$. Since $r_1, \hdots, r_{n+u}$ are independent, we have
\begin{equation*} 
\begin{split}
& \EE( \# \Sur(F_n/\left< r_i \right>, H_1) \# \Sur(F_n/\left< r_i b_i \right>, H_2) ) \\
= \, & \sum_{\substack{\phi_1 \in \Sur(F_n, H_1) \\ \phi_2 \in \Sur(F_n, H_2)}} \prod_{i=1}^{n+u} \PP(\phi_1 (r_i) = 1 \text{ and } \phi_2 (r_i) = \phi_2(b_i)^{-1}).
\end{split}    
\end{equation*}
Assume that there are $r_1, \hdots, r_{n+u} \in F_n$ such that
$\phi_1 (r_i) = 1$ and $\phi_2 (r_i) = \phi_2(b_i)^{-1}$ for each $i$. Then $\phi_2(b_i) = \phi_2(r_i)^{-1} \in \phi_2(\ker \phi_1)$ so $\phi_2(\left< b_i \right>) \leq \phi_2(\ker \phi_1)$. 
Similarly, $\phi_1(r_ib_i) = \phi_1(b_i)$ and $\phi_2(r_ib_i)=1$ so $\phi_1(\left< b_i \right>) \leq \phi_1(\ker \phi_2)$. 
If we define
$$
S_{G_1, G_2}^b := \left\{ (\phi_1, \phi_2) \in S_{G_1, G_2} \mid \phi_1(\left< b_i \right>) \leq G_1 \text{ and } \phi_2(\left< b_i \right>) \leq G_2 \right\},
$$
then we have
\begin{align*} 
& \sum_{\substack{\phi_1 \in \Sur(F_n, H_1) \\ \phi_2 \in \Sur(F_n, H_2)}} \prod_{i=1}^{n+u} \PP(\phi_1 (r_i) = 1 \text{ and } \phi_2 (r_i) = \phi_2(b_i)^{-1}) \\
= \, & \sum_{\substack{G_1, G_2 \\ H_1/G_1 \cong H_2/G_2}} \sum_{(\phi_1, \phi_2) \in S_{G_1, G_2}^{b}} \prod_{i=1}^{n+u} \PP(\phi_1 (r_i) = 1 \text{ and } \phi_2 (r_i) = \phi_2(b_i)^{-1}) \\
= \, & \sum_{\substack{G_1, G_2 \\ H_1/G_1 \cong H_2/G_2}} \sum_{(\phi_1, \phi_2) \in S_{G_1, G_2}^{b}} \prod_{i=1}^{n+u} \frac{1}{\left| H_1 \right|} \PP( \phi_2 (k_i) = \phi_2(b_i)^{-1}) \;\; (k_i \in \ker \phi_1) \\
= \, & \sum_{\substack{G_1, G_2 \\ H_1/G_1 \cong H_2/G_2}} \frac{\left| S_{G_1, G_2}^{b} \right|}{\left| H_1 \right|^{n+u} \left| G_2 \right|^{n+u}}.
\end{align*}
Now we bound the size of the set $S_{G_1, G_2}^{b}$ for each $(G_1, G_2)$. Let $r = \rank(H_1)$.
\begin{enumerate}
    \item $(G_1, G_2)=(H_1, H_2)$. Since $S_{H_1, H_2}^{b} = S_{H_1, H_2}$ and
    $$
    \bigsqcup_{\substack{G_1, G_2 \\ H_1/G_1 \cong H_2/G_2}}  S_{G_1, G_2} = \Sur(F_n, H_1) \times \Sur(F_n, H_2),
    $$
    we can obtain a lower bound of $\left| S_{H_1, H_2}^{b} \right|$ from upper bounds of $\left| S_{G_1, G_2} \right|$ for every $(G_1, G_2) \neq (H_1, H_2)$ such that $H_1/G_1 \cong H_2/G_2$. By Lemma \ref{lem4b}, we have
\begin{align*}
& \, \sum_{\substack{(G_1, G_2) \neq (H_1, H_2) \\ H_1/G_1 \cong H_2/G_2}} \left| S_{G_1, G_2} \right| \\
\leq & \, \sum_{\substack{(G_1, G_2) \neq (H_1, H_2) \\ H_1/G_1 \cong H_2/G_2}} \left| H_1 \right|^n \left| H_2 \right|^{n+r} \left ( \frac{\left| G_2 \right|}{\left| H_2 \right|} \right)^n \\
\leq & \, \sum_{\substack{(G_1, G_2) \neq (H_1, H_2) \\ H_1/G_1 \cong H_2/G_2}} \frac{\left| H_1 \right|^n \left| H_2 \right|^{n+r}}{2^n} \\
= & \, O_{H_1, H_2}\left ( \frac{\left| H_1 \right|^n \left| H_2 \right|^{n}}{2^n} \right ).
\end{align*}
Thus we have
$$
\lim_{n \rightarrow \infty} \sum_{\substack{(G_1, G_2) \neq (H_1, H_2) \\ H_1/G_1 \cong H_2/G_2}} \frac{\left| S_{G_1, G_2} \right|}{\left| H_1 \right|^{n+u} \left| H_2 \right|^{n+u}} = 0
$$
and
\begin{equation} \label{eq42d}
\begin{split}
& \lim_{n \rightarrow \infty} \frac{\left| S_{H_1, H_2}^b \right|}{\left| H_1 \right|^{n+u} \left| H_2 \right|^{n+u}} \\
= \, & \lim_{n \rightarrow \infty} \frac{\left| \Sur(F_n, H_1) \times \Sur(F_n, H_2) \right|}{\left| H_1 \right|^{n+u} \left| H_2 \right|^{n+u}} \\
= \, & \frac{1}{\left| H_1 \right|^{u} \left| H_2 \right|^{u}}.
\end{split}
\end{equation}
   
    \item $(G_1, G_2) \neq (H_1, H_2)$. Choose a basis $\left\{ y_1, \hdots, y_{n} \right\}$ of $F_n$ such that 
    $y_i \in \left< b_{1}, \hdots, b_{n+u} \right>$ for $1 \leq i \leq d_n$. 
    If $(\phi_1, \phi_2) \in S_{G_1, G_2}^b$, then $\phi_1$ is determined by $\phi_1(y_1), \hdots, \phi_1(y_{d_n}) \in G_1$ and $\phi_1(y_{d_n+1}), \hdots, \phi_1(y_n) \in H_1$ so 
    \begin{equation} \label{eq42e}
    \# \begin{Bmatrix}
\phi_1 \in \Sur(F_n, H_1) \mid (\phi_1, \phi_2) \in S_{G_1, G_2}^b \\
\text{for some } \phi_2 \in \Sur(F_n, H_2)
\end{Bmatrix}
    \leq \left| G_1 \right|^{d_n} \left| H_1 \right|^{n-d_n}.
    \end{equation}
    Since $S_{G_1, G_2}^b$ is a subset of $S_{G_1, G_2}$, for a fixed $\phi_1 \in \Sur(F_n, H_1)$ we have
    \begin{equation} \label{eq42f}
    \# \left\{ \phi_2 \in \Sur(F_n, H_2) \mid (\phi_1, \phi_2) \in S_{G_1, G_2}^b \right\} \leq \left| G_2 \right|^{n} \left| H_2 \right|^{r}
    \end{equation}
    by the proof of Lemma \ref{lem4b}. The equations (\ref{eq42e}) and (\ref{eq42f}) imply that
    \begin{align*}
    & \sum_{\substack{(G_1, G_2) \neq (H_1, H_2) \\ H_1/G_1 \cong H_2/G_2}} \frac{\left| S_{G_1, G_2}^b \right|}{\left| H_1 \right|^{n+u} \left| G_2 \right|^{n+u}} \\
\leq \, & \sum_{\substack{(G_1, G_2) \neq (H_1, H_2) \\ H_1/G_1 \cong H_2/G_2}} \frac{\left| G_1 \right|^{d_n} \left| H_1 \right|^{n-d_n} \left| G_2 \right|^{n} \left| H_2 \right|^{r}}{\left| H_1 \right|^{n+u} \left| G_2 \right|^{n+u}} \\
= \, & \sum_{\substack{(G_1, G_2) \neq (H_1, H_2) \\ H_1/G_1 \cong H_2/G_2}} \frac{\left| G_1 \right|^{d_n} \left| H_2 \right|^{r}}{\left| H_1 \right|^{d_n+u} \left| G_2 \right|^{u}} \\
= \, & O_{H_1, H_2} \left ( \frac{1}{\left| H_1 \right|^{d_n}} \left ( \frac{\left| H_1 \right|}{2} \right )^{d_n} \right ) \\
= \, & O_{H_1, H_2} \left ( \frac{1}{2^{d_n}} \right ).
    \end{align*}
By the assumption, we have $\lim_{n \rightarrow \infty} d_n = \infty$ so
\begin{equation} \label{eq42g}
\lim_{n \rightarrow \infty} \sum_{\substack{(G_1, G_2) \neq (H_1, H_2) \\ H_1/G_1 \cong H_2/G_2}} \frac{\left| S_{G_1, G_2}^b \right|}{\left| H_1 \right|^{n+u} \left| G_2 \right|^{n+u}} = 0.
\end{equation}
\end{enumerate}
We conclude the proof from the equations (\ref{eq42d}) and (\ref{eq42g}).
\end{proof}

By applying Theorem \ref{thm1i} to the above proposition, we can conclude that the distributions of the random groups $F_n/\left< r_i \right>$ and $F_n / \left< r_i b_{n, i} \right>$ in $\GG$ are asymptotically independent.

\begin{theorem} \label{thm4d}
Let $u$, $r_i$, $b_{n, i}$ and $d_n$ be as in Proposition \ref{prop4c}. Then the joint distributions of
$$
(F_n/\left< r_1, \hdots, r_{n+u} \right>, \, F_n/\left< r_1 b_{n, 1}, \hdots, r_{n+u} b_{n, n+u} \right>)
$$
weakly converge in distribution to the probability measure $\mu_u \times \mu_u$ on $\GG^2$ as $n \rightarrow \infty$.
\end{theorem}

\begin{proof}
Let $H_1$ and $H_2$ be finite groups. By Lemma \ref{lem4a} and Proposition \ref{prop4c}, we have
\begin{equation*}
\begin{split}
& \lim_{n \rightarrow \infty} \EE( \# \Sur(F_n/\left< r_i \right>, H_1) \# \Sur(F_n/\left< r_i b_{n, i} \right>, H_2) ) \\
= \, & \EE( \# \Sur(Y_1, H_1) \# \Sur(Y_2, H_2) ) \\
= \, & \frac{1}{\left| H_1 \right|^{u} \left| H_2 \right|^{u}}
\end{split}
\end{equation*}
where $Y_1, Y_2 \in \GG$ are independent random groups following the distribution of $X_u$. Now Theorem \ref{thm1i} (for the case $\Gamma = 1$) finishes the proof.
\end{proof}

\section*{Acknowledgments}
This work was supported by the new faculty research fund of Ajou University (S-2023-G0001-00236). We thank Hoi H. Nguyen and Roger Van Peski for helpful comments. 

\end{document}